\documentclass[12pt,a4paper,twoside,english,american]{scrartcl}
\usepackage{helvet}
\usepackage[latin9]{inputenc}
\usepackage{babel}
\usepackage{verbatim}
\usepackage{amsthm}
\usepackage{amsmath}
\usepackage{amssymb}
\usepackage[unicode=true,pdfusetitle,
 bookmarks=true,bookmarksnumbered=false,bookmarksopen=false,
 breaklinks=false,pdfborder={0 0 1},backref=false,colorlinks=false]
 {hyperref}

\makeatletter

\numberwithin{equation}{section}
\numberwithin{figure}{section}
 \theoremstyle{definition}
 \newtheorem*{defn*}{\protect\definitionname}
\theoremstyle{plain}
\newtheorem{thm}{\protect\theoremname}[section]
  \theoremstyle{remark}
  \newtheorem{rem}[thm]{\protect\remarkname}
  \theoremstyle{plain}
  \newtheorem{prop}[thm]{\protect\propositionname}
  \theoremstyle{plain}
  \newtheorem{cor}[thm]{\protect\corollaryname}
  \theoremstyle{plain}
  \newtheorem{lem}[thm]{\protect\lemmaname}
  \theoremstyle{definition}
  \newtheorem{defn}[thm]{\protect\definitionname}

\usepackage[T1]{fontenc}    % Silbentrennung auch nach Umlauten mglich

\usepackage{a4wide}        % A4-Format
\usepackage{tu-preprint}
\usepackage{amsmath}
\usepackage{enumerate}

\usepackage{euscript}
\usepackage{mathtools}
\allowdisplaybreaks[1] % allow page breaks in some displayed equations
\usepackage{amsfonts}
\newenvironment{keywords}{ \noindent\footnotesize\textbf{Keywords and phrases:}}{}

\newenvironment{class}{\noindent\footnotesize\textbf{Mathematics subject classification 2010:}}{}

\theoremstyle{definition}
\newtheorem*{hyp}{Hypotheses}

\usepackage{color,tu-preprint}

\newcommand*{\dive}{\operatorname{div}}

\newcommand*{\Grad}{\operatorname{Grad}}
\newcommand*{\Div}{\operatorname{Div}}

\renewcommand*{\i}{\mathrm{i}}
\DeclareMathAccent{\Circ}{\mathalpha}{operators}{"17}
\newcommand{\interior}[1]{\Circ{#1}}
\renewcommand{\Im}{\operatorname{\mathfrak{Im}}}
\renewcommand{\Re}{\operatorname{\mathfrak{Re}}}

\newcommand{\oi}[2]{\left]#1,#2 \right[}
\newcommand{\rga}[1]{\left]#1,\infty  \right[}

\newcommand{\X}{\raisebox{2pt}{$\chi$}}

\renewcommand{\tilde}{\widetilde}
\renewcommand*{\epsilon}{\varepsilon}
\renewcommand*{\theta}{\vartheta}
\renewcommand*{\rho}{\varrho}

\makeatother

  \addto\captionsamerican{\renewcommand{\corollaryname}{Corollary}}
  \addto\captionsamerican{\renewcommand{\definitionname}{Definition}}
  \addto\captionsamerican{\renewcommand{\lemmaname}{Lemma}}
  \addto\captionsamerican{\renewcommand{\propositionname}{Proposition}}
  \addto\captionsamerican{\renewcommand{\remarkname}{Remark}}
  \addto\captionsamerican{\renewcommand{\theoremname}{Theorem}}
  \addto\captionsenglish{\renewcommand{\corollaryname}{Corollary}}
  \addto\captionsenglish{\renewcommand{\definitionname}{Definition}}
  \addto\captionsenglish{\renewcommand{\lemmaname}{Lemma}}
  \addto\captionsenglish{\renewcommand{\propositionname}{Proposition}}
  \addto\captionsenglish{\renewcommand{\remarkname}{Remark}}
  \addto\captionsenglish{\renewcommand{\theoremname}{Theorem}}
  \providecommand{\corollaryname}{Corollary}
  \providecommand{\definitionname}{Definition}
  \providecommand{\lemmaname}{Lemma}
  \providecommand{\propositionname}{Proposition}
  \providecommand{\remarkname}{Remark}
\providecommand{\theoremname}{Theorem}

\begin{document}
\selectlanguage{english}%
\institut{Institut f\"ur Analysis}

\preprintnumber{MATH-AN-01-2013}

\preprinttitle{On Non-Autonomous Evolutionary Problems.}

\author{Rainer Picard, Sascha Trostorff, Marcus Waurick \& Maria Wehowski} 

\makepreprinttitlepage

\selectlanguage{american}%
\setcounter{section}{-1}

\date{}

\title{On Non-Autonomous Evolutionary Problems.}

\author{Rainer Picard, Sascha Trostorff, Marcus Waurick, Maria Wehowski\\
Institut f\"ur Analysis, Fachrichtung Mathematik\\
Technische Universit\"at Dresden\\
Germany\\
rainer.picard@tu-dresden.de\\
sascha.trostorff@tu-dresden.de\\
marcus.waurick@tu-dresden.de\\
maria.wehowski@mailbox.tu-dresden.de }
\maketitle
\begin{abstract}
The paper extends well-posedness results of a previously explored
class of time-shift invariant evolutionary problems to the case of
non-autonomous media. The Hilbert space setting developed for the
time-shift invariant case can be utilized to obtain an elementary
approach to non-autonomous equations. The results cover a large class
of evolutionary equations, where well-known strategies like evolution
families may be difficult to use or fail to work. We exemplify the
approach with an application to a Kelvin-Voigt-type model for visco-elastic
solids.
\end{abstract}
\begin{keywords}
non-autonomous, evolutionary problems, extrapolation spaces, Sobolev lattices  \end{keywords}

\begin{class}
35F05, 35F10, 35F15, 37L05,74B99
\end{class}

\newpage

\tableofcontents{}

\section{Introduction}

In a number of studies it has been demonstrated that systems of the
form
\begin{equation}
\left(\partial_{0}\mathcal{M}+A\right)u=F,\label{eq:evo-eq}
\end{equation}
where $\mathcal{\ensuremath{M}}$ is a continuous, linear mapping
and the densely defined, closed linear operator $A$ is such that
$A$ and $A^{*}$ are maximal $\omega$-accretive for some suitable
$\omega\in\mathbb{R}$, cover numerous models from mathematical physics.
Indeed, $A$ skew-selfadjoint%
\footnote{Two densely defined linear operators $A,B$ are skew-adjoint (to each
other) if $A=-B^{*}$. If $A=B$, we call $A$ skew-selfadjoint (rather
than self-skew-adjoint). The proper implications 
\[
A\mbox{ selfadjoint}\implies A\mbox{ symmetric}\implies A\mbox{ Hermitean}
\]
 are paralleled by 
\[
A\mbox{ skew-selfadjoint}\implies A\mbox{ skew-symmetric}\implies A\mbox{ skew-Hermitean}.
\]
Frequently, ``skew-adjoint'' is used to mean ``skew-selfadjoint''.
We shall, however, not follow this custom for the obvious reason.%
} is a standard situation, which for simplicity we shall assume throughout.
The well-posedness of (\ref{eq:evo-eq}) hinges on a positive definiteness
assumption imposed on $\partial_{0}\mathcal{M}$ in a suitable space-time
Hilbert space setting. Under this assumption the solution theory is
comparatively elementary since $\left(\partial_{0}\mathcal{M}+A\right)$
together with its adjoint are positive definite yielding that $\left(\partial_{0}\mathcal{M}+A\right)$
has dense range and a continuous inverse.

In applications of this setting the operator $A$ has a rather simple
structure whereas the complexity of the physical system is encoded
in the ``material law'' operator $\mathcal{M}$. A simple but important
case is given by 
\[
\mathcal{M}=M_{0}+\partial_{0}^{-1}M_{1},
\]
where $M_{0}$, $M_{1}$ are time-independent continuous linear operators.
Here we have anticipated that in the Hilbert space setting to be constructed
$\partial_{0}^{-1}$ (time integration) has a proper meaning. The
positive definiteness assumption requires $M_{0}$ to be non-negative
and selfadjoint and
\begin{equation}
\rho M_{0}+\Re M_{1}\geq c_{0}\label{eq:pos-def-a}
\end{equation}
for some $c_{0}\in\rga0$ and all sufficiently large $\rho\in\rga0$.
Since we do not assume that $M_{0}$ is always strictly positive,
(\ref{eq:pos-def-a}) may imply constraints on $M_{1}$. If $M_{0}$
is positive definite it may seem natural, following the proven idea
of first finding a fundamental solution (given by an associated semi-group),
and then to obtain general solutions as convolutions with the data
(Duhamel's principle, variation of constants formula) and so proving
well-posedness. This is the classical method of choice in a Banach
space setting, see e.g.\ \cite{0978.34001,EngNag,Ref189,Ref497,Yosida_6th}
as general references. In comparison our approach is (currently) limited
to a Hilbert space setting, however, apart from being conceptually
more elementary, it allows to incorporate delay and convolution integral
terms by a simple perturbation argument and, if $M_{0}$ has a non-trivial
kernel, the system becomes a differential-algebraic systems, which
to the above approach makes no difference, but cannot be conveniently
analyzed within the framework of semi-group theory. 

The purpose of this paper is to extend well-posedness results previously
obtained for time-shift invariant material operators $\mathcal{M}$
to cases, where $\mathcal{M}$ is not time-shift invariant. This is
the so-called time-dependent or non-autonomous case. The above-mentioned
limitations of the semi-group approach carry over to the application
of classical strategies based on evolution families intoduced by Kato
(\cite{Kato 1953}), for a survey see e.g.\ \cite{Ref647,Ref497},
which are the corresponding abstract Green's functions, in the non-autonomous
case. The approach we shall develop here, by-passes the relative sophistication
of the classical approach based on evolution families and extends,
moreover, to differential-algebraic cases and allows to include memory
effects, in a simple unified setting.

To keep the presentation self-contained we construct the Hilbert space
setting in sufficient detail and formulate our results so that the
autonomous case re-appears as a special case of the general non-autonomous
situation.

In order to formulate the problem class rigorously and to avoid at
the same time to incur unnecessary regularity constraints on data
and domain for prospective applications, it is helpful to introduce
suitable extrapolation Hilbert spaces (Sobolev lattices). This will
be done in the next section. In Section \ref{sec:Space-Time-Evolution-Equations}
we shall describe a class of non-autonomous evolutionary equations
and its solution theory. The paper concludes with the discussion of
a particular application to a class of Kelvin-Voigt-type models for
visco-elastic solids to exemplify the theoretical findings.

In the following let $H$ be a Hilbert space with inner product $\langle\cdot|\cdot\rangle$
and induced norm $|\cdot|$.

\section{Preliminaries}

In this subsection we recall the construction of a short Sobolev chain
associated with a normal, boundedly invertible operator $N$ on some
Hilbert space $H$ as it was presented for instance in \cite[Section 2.1]{Picard_McGhee}.
We note that the construction can be generalized to the case of a
closed, boundedly invertible operator on some Hilbert space. However,
since the operators we are interested in are all normal, we may reduce
ourselves to the slightly easier case of normal operators. 
\begin{defn*}
Let $H$ be a Hilbert space and $N\colon D(N)\subseteq H\to H$ be
a normal operator with $0\in\rho(N).$ Then the sesqui-linear form
\begin{align*}
\langle\cdot|\cdot\rangle_{H_{1}(N)}\colon D(N)\times D(N) & \to\mathbb{C}\\
(f,g) & \mapsto\langle Nf|Ng\rangle,
\end{align*}
defines an inner product on $D(N)$ and due to the closedness of $N,$
$D(N)$ equipped with this inner product becomes a Hilbert space.
We denote it by $H_{1}(N).$ Moreover, by 
\begin{align*}
\langle\cdot|\cdot\rangle_{H_{-1}(N)}\colon H\times H & \to\mathbb{C}\\
(f,g) & \mapsto\langle N^{-1}f|N^{-1}g\rangle
\end{align*}
we define an inner product on $H$ and we denote the completion of
$H$ with respect to the induced norm by $H_{-1}(N).$ The triple
$\left(H_{1}(N),H_{0}(N),H_{-1}(N)\right)$ is called \emph{short
Sobolev chain associated with $N$, }where $H_{0}(N)\coloneqq H$.\end{defn*}
\begin{rem}
Note that the above construction can be done analogously for the $k$-th
power of $N$ for $k\in\mathbb{N}.$ The resulting sequence of Hilbert
spaces $\left(H_{k}(N)\right)_{k\in\mathbb{Z}},$ where we set $H_{k}(N)\coloneqq H_{1}\left(N^{k}\right)$
and $H_{-k}(N)\coloneqq H_{-1}\left(N^{k}\right)$ for $k\in\mathbb{N}\setminus\left\{ 0\right\} $
is called the \emph{Sobolev chain associated with $N$}. 
\end{rem}
A simple estimate shows that $H_{1}(N)$ is densely and continuously
embedded into $H_{0}(N),$ which itself is densely and continuously
embedded into $H_{-1}(N).$ Thus we arrive at a chain of dense and
continuous embeddings of the form 
\[
H_{1}(N)\hookrightarrow H_{0}(N)\hookrightarrow H_{-1}(N).
\]
We can establish the operator $N$ as a unitary operator acting on
this chain.
\begin{prop}[{\cite[Theorem 2.1.6]{Picard_McGhee}}]
\label{prop:realization} The operator $N\colon H_{1}(N)\to H_{0}(N)$
is unitary. Moreover, the operator $N\colon D(N)\subseteq H_{0}(N)\to H_{-1}(N)$
has a unitary extension to $H_{0}(N).$%
\footnote{We will not distinguish notationally between the operator $N$ and
its unitary realizations on the Sobolev chain associated with $N$.%
} 
\end{prop}
Since $N$ is boundedly invertible its adjoint $N^{\ast}$ has a bounded
inverse, too. Hence we can do the same construction of a Sobolev chain
associated with $N^{\ast}.$ However, since $N$ is assumed to be
normal we have that $D(N)=D(N^{\ast}),$ as well as $|Nx|=|N^{\ast}x|$
for each $x\in D(N).$ This yields that the Sobolev chains for $N$
and $N^{\ast}$ coincide. Hence, by Proposition \ref{prop:realization}
we obtain the following result.
\begin{cor}
The operator $N^{\ast}\colon D(N^{\ast})\subseteq H_{k}(N)\to H_{k-1}(N)$
has a unitary extension to $H_{k}(N)$ for $k\in\{0,1\}.$ 
\end{cor}
Again this realization will be denoted by $N^{\ast}$ although this
might cause confusion, since $N^{\ast}$ could be interpreted as the
adjoint of the unitary realization of the operator $N$ as in Proposition
\ref{prop:realization}. The adjoint would then be the respective
inverse.

Note that for normal $N\in L\left(H,H\right)\eqqcolon L(H)$, the
space of continuous linear mappings from $H$ to $H$, we have that
\[
H_{k}\left(N\right)=H
\]
as topological linear spaces with merely different inner products
(inducing equivalent norms) in the different Hilbert spaces $H_{k}\left(N\right)$,
$k\in\{-1,0,1\}$. This indicates that considering continuous linear
operators $N$ does not lead to interesting chains.

In the remaining part of this section we consider a particular example
of a normal operator and its associated Sobolev chain, namely the
time derivative in an exponentially weighted $L^{2}$-space (see \cite{Picard_McGhee,picard1989hilbert,Kalauch2011}
for more details). \\

For $\rho\in\mathbb{R}$ we consider the Hilbert space 
\[
H_{\rho,0}(\mathbb{R})\coloneqq\left\{ \left.f\in L_{\mathrm{loc}}^{2}(\mathbb{R})\,\right|\,\left(x\mapsto\exp(-\rho x)f(x)\right)\in L^{2}(\mathbb{R})\right\} 
\]
equipped with the inner product 
\[
\langle f|g\rangle_{H_{\rho,0}(\mathbb{R})}\coloneqq\intop_{\mathbb{R}}f(x)^{\ast}g(x)\exp(-2\rho x)\mbox{ d}x\quad\left(f,g\in H_{\rho,0}(\mathbb{R})\right).
\]
We define the operator $\partial_{0,\rho}$ on $H_{\rho,0}(\mathbb{R})$
as the closure of 
\begin{align*}
\partial_{0,\rho}|_{\interior C_{\infty}(\mathbb{R})}\colon\interior C_{\infty}(\mathbb{R})\subseteq H_{\rho,0}(\mathbb{R}) & \to H_{\rho,0}(\mathbb{R})\\
\phi & \mapsto\phi',
\end{align*}
where by $\interior C_{\infty}(\mathbb{R})$ we denote the space of
arbitrarily often differentiable functions on $\mathbb{R}$ with compact
support%
\footnote{The domain of $\partial_{0,\rho}$ consists precisely of the functions
$f\in H_{\rho,0}(\mathbb{R})$ with distributional derivative lying
in $H_{\rho,0}(\mathbb{R})$.%
}. In this way we obtain a normal operator with $\Re\partial_{0,\rho}=\rho.$
Hence, for $\rho\ne0$ the operator $\partial_{0,\rho}$ is boundedly
invertible and one can show that $\left\Vert \partial_{0,\rho}^{-1}\right\Vert _{L(H_{\rho,0}(\mathbb{R}))}=1/|\rho|$.
Thus for $\rho\ne0$ we can construct the Sobolev chain associated
with $\partial_{0,\rho}$ and we introduce the notation $H_{\rho,k}(\mathbb{R})\coloneqq H_{k}\left(\partial_{0,\rho}\right)$
for $\rho\ne0$ and $k\in\{-1,0,1\}.$ For $\Im\partial_{0,\rho}$
we have as a spectral representation the Fourier-Laplace transform
$\mathcal{L}_{\rho}\colon H_{\rho,0}(\mathbb{R})\to L^{2}\left(\mathbb{R}\right)$
given by the unitary extension of 
\begin{align*}
\interior C_{\infty}\left(\mathbb{R}\right)\subseteq H_{\rho,0}(\mathbb{R}) & \to L^{2}(\mathbb{R})\\
\phi & \mapsto\left(x\mapsto\frac{1}{\sqrt{2\pi}}\int_{\mathbb{R}}\exp\left(-\mathrm{i}xy\right)\;\exp\left(-\rho y\right)\phi(y)\; dy\right).
\end{align*}
In other words, we have the unitary equivalence
\[
\Im\partial_{0,\rho}=\mathcal{L}_{\rho}^{\ast}m\mathcal{L}_{\rho},
\]
where $m$ denotes the multiplication-by-argument operator in $L^{2}\left(\mathbb{R}\right)$
with maximal domain, i.e.\ $\left(mf\right)(x)\coloneqq xf(x)$ for
a.e.\  $x\in\mathbb{R},\, f\in D(m)\coloneqq\left\{ g\in L^{2}(\mathbb{R})\,|\,(x\mapsto xg(x))\in L^{2}(\mathbb{R})\right\} $.
The latter yields 
\[
\partial_{0,\rho}=\mathcal{L}_{\rho}^{\ast}\left(\i m+\rho\right)\mathcal{L}_{\rho}.
\]
For $\rho\ne0$ we can represent the resolvent $\partial_{0,\rho}$
as an integral operator given by 
\[
\left(\partial_{0,\rho}^{-1}u\right)\left(x\right)=\int_{-\infty}^{x}u\left(t\right)\:\mathrm{d}t\quad(x\in\mathbb{R}\:\mbox{a.e.})
\]
if $\rho>0$ and 
\[
\left(\partial_{0,\rho}^{-1}u\right)\left(x\right)=-\int_{x}^{\infty}u\left(t\right)\:\mathrm{d}t\quad(x\in\mathbb{R}\:\mbox{a.e.})
\]
if $\rho<0$ for all $u\in H_{\rho,0}(\mathbb{R}).$ Since we are
interested in the (forward) causal situation (see Definition \ref{def: causality}
below), we assume $\rho>0$ throughout. Moreover, in the following
we shall mostly write $\partial_{0}$ for $\partial_{0,\rho}$ if
the choice of $\rho$ is clear from the context.

Let now $N$ denote a normal operator in a Hilbert space $H$ with
$0$ in its resolvent set. Then $N$ has a canonical extension to
the time-dependent case, i.e., to $H_{\rho,0}\left(\mathbb{R};H\right)\cong H_{\rho,0}(\mathbb{R})\otimes H$,
the space of $H$-valued functions on $\mathbb{R}$, which are square-integrable
with respect to the exponentially weighted Lebesgue-measure. Analogously
we can extend $\partial_{0}$ to an operator on $H_{\rho,0}(\mathbb{R};H)$
in the canonical way. Then $\partial_{0}$ and $N$ become commuting
normal operators and by combining the two chains we obtain a Sobolev
lattice in the sense of \cite[Sections 2.2 and 2.3]{Picard_McGhee}
based on $\left(\partial_{0},N\right)$ yielding a family of Hilbert
spaces 
\[
\left(H_{\rho,k}\left(\mathbb{R};H_{s}\left(N\right)\right)\right)_{k,s\in\{-1,0,1\}}
\]
with norms $\left|\:\cdot\:\right|_{\rho,k,s}$ given by 
\[
v\mapsto\left|\partial_{0}^{k}N^{s}v\right|_{H_{\rho,0}(\mathbb{R};H)}
\]

for $k,s\in\{-1,0,1\}$. The operators $\partial_{0}$ and $N$ can
then be established as unitary mappings from $H_{\rho,k}\left(\mathbb{R};H_{s}(N)\right)$
to $H_{\rho,k-1}\left(\mathbb{R};H_{s}(N)\right)$ for $k\in\{0,1\},s\in\{-1,0,1\}$
and from $H_{\rho,k}\left(\mathbb{R};H_{s}(N)\right)$ to $H_{\rho,k}\left(\mathbb{R};H_{s-1}(N)\right)$
for $k\in\{-1,0,1\},s\in\{0,1\},$ respectively.

\section{Space-time evolutionary equations\label{sec:Space-Time-Evolution-Equations} }

\subsection*{Well-posedness for a class of evolutionary problems}

We are now ready to rigorously approach the well-posedness class we
wish to present. We shall consider equations of the form
\[
\left(\partial_{0}M(m_{0},\partial_{0}^{-1})+A\right)u=F,
\]
where for simplicity we assume that $A$ is skew-selfadjoint in a
Hilbert space $H$ and $M(t,\cdot)$ is a material law function in
the sense of \cite{Picard} for almost every $t\in\mathbb{R}$. More
specifically we assume that $M$ is of the form%
\footnote{Note that $\left(\Phi\left(m_{0}\right)\varphi\right)(t)=\Phi(t)\varphi(t)$
for $t\in\mathbb{R},\varphi\in\interior C_{\infty}(\mathbb{R};H)$
and $\Phi\colon\mathbb{R}\to L(H)$ strongly measurable and bounded.
Hence, $\Phi\left(m_{0}\right)\in L\left(H_{\rho,0}(\mathbb{R};H)\right)$
and $\left\Vert \Phi(m_{0})\right\Vert _{L\left(H_{\rho,0}(\mathbb{R};H)\right)}\leq\sup_{t\in\mathbb{R}}\left\Vert \Phi(t)\right\Vert _{L(H)}$
for each $\rho\geq0$. %
}
\[
M(m_{0},\partial_{0}^{-1})=M_{0}(m_{0})+\partial_{0}^{-1}M_{1}(m_{0}),
\]

where $M_{0},M_{1}\in L_{s}^{\infty}(\mathbb{R};L(H))$, the space
of strongly measurable uniformly bounded functions with values in
$L(H)$. We understand $\partial_{0}M_{0}(m_{0})+M_{1}(m_{0})+A$
as an (unbounded) operator in $H_{\rho,0}(\mathbb{R};H)$ with maximal
domain
\[
\mathcal{D}\coloneqq\{u\in H_{\rho,0}(\mathbb{R};H)\,|\,\partial_{0}M_{0}(m_{0})u+Au\in H_{\rho,0}(\mathbb{R};H)\}.
\]
Note that for $u\in H_{\rho,0}(\mathbb{R};H)$ we have
\[
\partial_{0}M_{0}(m_{0})u+M_{1}(m_{0})u+Au\in H_{\rho,-1}(\mathbb{R};H_{-1}(A+1))
\]
and for this to be in $H_{\rho,0}(\mathbb{R};H)$ is the constraint
determining the maximal domain $\mathcal{D}$. 

\begin{hyp} We say that $T\in L_{s}^{\infty}(\mathbb{R};L(H))$ satisfies
the property

\begin{enumerate}[(a)]

\item \label{selfadjoint} if $T(t)$ is selfadjoint $(t\in\mathbb{R})$,

\item \label{non-negative} if $T(t)$ is non-negative $(t\in\mathbb{R})$,

\item  \label{Lipschitz} if the mapping $T$ is Lipschitz-continuous,
where we denote the smallest Lipschitz-constant of $T$ by $|T|_{\mathrm{Lip}}$,
and

\item  \label{differentiable} if there exists a set $N\subseteq\mathbb{R}$
of measure zero such that for each $x\in H$ the function 
\[
\mathbb{R}\setminus N\ni t\mapsto T(t)x
\]
is differentiable%
\footnote{If $H$ is separable, then the strong differentiability of $T$ on
$\mathbb{R}\setminus N$ for some set $N$ of measure zero already
follows from the Lipschitz-continuity of $T.$%
}. 

\end{enumerate}

\end{hyp}
\begin{lem}
\label{lem:product rule} Assume that $M_{0}$ satisfies properties
(\ref{Lipschitz}) and (\ref{differentiable}). Then for each $t\in\mathbb{R}$
the mapping 
\begin{eqnarray*}
\dot{M}_{0}(t)\colon H & \to & H\\
x & \mapsto & \begin{cases}
\left(M_{0}(\cdot)x\right)'(t), & t\in\mathbb{R}\setminus N,\\
0, & t\in N
\end{cases}
\end{eqnarray*}
is a bounded linear operator with $\|\dot{M}_{0}(t)\|_{L(H)}\leq|M_{0}|_{\mathrm{Lip}}.$
Thus, $\dot{M}_{0}\in L_{s}^{\infty}(\mathbb{R};L(H))$ gives rise
to a multiplication operator $\dot{M}_{0}(m_{0})\in L(H_{\rho,0}(\mathbb{R};H))$
given by 
\[
(\dot{M}_{0}(m_{0})u)(t)\coloneqq\left(M_{0}(\cdot)u(t)\right)'(t)
\]
for $u\in H_{\rho,0}(\mathbb{R};H)$ and almost every $t\in\mathbb{R}.$
Moreover, for $u\in D(\partial_{0})$ the product rule 
\begin{equation}
\partial_{0}M_{0}(m_{0})u=\dot{M}_{0}(m_{0})u+M_{0}(m_{0})\partial_{0}u\label{eq:product_rule}
\end{equation}
holds. In particular, $M_{0}(m_{0})\in L\left(H_{\rho,-1}(\mathbb{R};H)\right).$
If, in addition, $M_{0}$ satisfies property (\ref{selfadjoint})
then $\dot{M}_{0}(m_{0})$ is selfadjoint.\end{lem}
\begin{proof}
Let $t\in\mathbb{R}\setminus N$. The linearity of $\dot{M}_{0}(t)$
is obvious. For $x\in H$ we estimate 
\[
\frac{1}{|h|}\left|\left(M_{0}(t+h)-M_{0}(t)\right)x\right|\leq|M_{0}|_{\mathrm{Lip}}|x|
\]
for each $h\in\mathbb{R}\setminus\{0\}.$ Thus, $|\dot{M}_{0}(t)x|=\left|\left(M_{0}(\cdot)x\right)'(t)\right|\leq|M_{0}|_{\mathrm{Lip}}|x|,$
which shows that $\dot{M}_{0}(t)\in L(H)$ with $\|\dot{M}_{0}(t)\|_{L(H)}\leq|M_{0}|_{\mathrm{Lip}}.$
Assuming property (a), we see that the selfadjointness of $\dot{M}_{0}(t)$
follows from 
\begin{eqnarray*}
\langle\dot{M}_{0}(t)x|y\rangle & = & \lim_{h\to0}\frac{1}{h}\langle\left(M_{0}(t+h)-M_{0}(t)\right)x|y\rangle\\
 & = & \lim_{h\to0}\frac{1}{h}\left\langle x\left|\left(M_{0}(t+h)-M_{0}(t)\right)y\right.\right\rangle \\
 & = & \langle x|\dot{M}_{0}(t)y\rangle
\end{eqnarray*}
for each $x,y\in H$. It is left to show the product rule (\ref{eq:product_rule}).
To this end, let $\phi\in\interior C_{\infty}(\mathbb{R};H).$ Then
we compute 
\begin{eqnarray*}
M_{0}(t+h)\phi(t+h)-M_{0}(t)\phi(t) & = & M_{0}(t+h)\left(\phi(t+h)-\phi(t)\right)+\left(M_{0}(t+h)-M_{0}(t)\right)\phi(t)
\end{eqnarray*}
for each $t,h\in\mathbb{R}.$ This yields
\begin{align*}
 & \frac{1}{h}\left(M_{0}(t+h)\phi(t+h)-M_{0}(t)\phi(t)\right)\\
 & =M_{0}(t+h)\left(\frac{1}{h}\left(\phi(t+h)-\phi(t)\right)\right)+\left(\frac{1}{h}\left(M_{0}(t+h)-M_{0}(t)\right)\right)\phi(t)
\end{align*}
 for every $t\in\mathbb{R},h\in\mathbb{R}\setminus\{0\}.$ The term
on the right-hand side in the latter formula tends to 
\[
M_{0}(t)\phi'(t)+\left(M_{0}(\cdot)\phi(t)\right)'(t)=\left(M_{0}(m_{0})\partial_{0}\phi\right)(t)+(\dot{M}_{0}(m_{0})\phi)(t)
\]
for $t\in\mathbb{R}\setminus N$ as $h\to0.$ Thus, the left-hand
side is differentiable almost everywhere and since $\dot{M}_{0}(m_{0})\phi+M_{0}(m_{0})\partial_{0}\phi\in H_{\rho,0}(\mathbb{R};H)$
we derive that $M_{0}(m_{0})\phi\in D(\partial_{0})$ and 
\[
\partial_{0}M_{0}(m_{0})\phi=\dot{M}_{0}(m_{0})\phi+M_{0}(m_{0})\partial_{0}\phi.
\]
The product rule (\ref{eq:product_rule}) for functions in $D(\partial_{0})$
now follows by approximation. To show that the operator $M_{0}(m_{0})$
can be established as a bounded operator on $H_{\rho,-1}(\mathbb{R};H)$
we observe that 
\[
\partial_{0}^{-1}M(m_{0})-M(m_{0})\partial_{0}^{-1}=\partial_{0}^{-1}\left(M(m_{0})\partial_{0}-\partial_{0}M(m_{0})\right)\partial_{0}^{-1}=-\partial_{0}^{-1}\dot{M}_{0}(m_{0})\partial_{0}^{-1}
\]
and thus, 
\begin{align*}
|M(m_{0})u|_{H_{\rho,-1}(\mathbb{R};H)} & =|\partial_{0}^{-1}M(m_{0})u|_{H_{\rho,0}(\mathbb{R};H)}\\
 & \leq|\partial_{0}^{-1}\dot{M}_{0}(m_{0})\partial_{0}^{-1}u|_{H_{\rho,0}(\mathbb{R};H)}+|M_{0}(m_{0})\partial_{0}^{-1}u|_{H_{\rho,0}(\mathbb{R};H)}\\
 & \leq\left(\frac{1}{\rho}|M_{0}|_{\mathrm{Lip}}+\|M_{0}(m_{0})\|_{L(H_{\rho,0}(\mathbb{R};H))}\right)|u|_{H_{\rho,-1}(\mathbb{R};H)}
\end{align*}
for each $u\in H_{\rho,0}(\mathbb{R};H).$\end{proof}
\begin{rem}
Note that the product rule 
\[
\partial_{0}M_{0}(m_{0})\phi=\dot{M}_{0}(m_{0})\phi+M_{0}(m_{0})\partial_{0}\phi
\]
for $\phi\in D(\partial_{0})$ can be extended by continuity to $\phi\in H_{\rho,0}(\mathbb{R};H)$.
Indeed, both the operators $\partial_{0}M_{0}(m_{0})$ and $\dot{M}_{0}(m_{0})+M_{0}(m_{0})\partial_{0}$
are densely defined continuous mappings from $H_{\rho,0}(\mathbb{R};H)$
to $H_{\rho,-1}(\mathbb{R};H)$ coinciding on the dense subset $H_{\rho,1}(\mathbb{R};H)\subseteq H_{\rho,0}(\mathbb{R};H)$. \end{rem}
\begin{cor}
\label{cor:domain-of-adjoint} Assume that $M_{0}$ satisfies properties
(\ref{Lipschitz}) and (\ref{differentiable}). Then 
\begin{eqnarray*}
\mathcal{D} & = & \left\{ u\in H_{\rho,0}(\mathbb{R};H)\,\left|\,\partial_{0}M_{0}(m_{0})u+Au\in H_{\rho,0}(\mathbb{R};H)\right.\right\} \\
 & = & \left\{ u\in H_{\rho,0}(\mathbb{R};H)\,\left|\, M_{0}(m_{0})\partial_{0}^{*}u-Au\in H_{\rho,0}(\mathbb{R};H)\right.\right\} .
\end{eqnarray*}
 Moreover, we have $(M_{0}(m_{0})\partial_{0}^{*}-A)u=(-\partial_{0}M_{0}(m_{0})+2\rho M_{0}(m_{0})+\dot{M}_{0}(m_{0})-A)u$
for all $u\in\mathcal{D}$.\end{cor}
\begin{proof}
Let $u\in H_{\rho,0}(\mathbb{R};H)$. Then $u\in\mathcal{D}$ if and
only if $(\partial_{0}M_{0}(m_{0})+A)u\in H_{\rho,0}(\mathbb{R};H)$.
Since $(-2\rho M_{0}(m_{0})-\dot{M}_{0}(m_{0}))u\in H_{\rho,0}(\mathbb{R};H)$,
we have that $(\partial_{0}M_{0}(m_{0})+A)u\in H_{\rho,0}(\mathbb{R};H)$
if and only if 
\begin{eqnarray*}
H_{\rho,0}(\mathbb{R};H) & \ni & (\partial_{0}M_{0}(m_{0})+A)u+(-2\rho M_{0}(m_{0})-\dot{M}_{0}(m_{0}))u\\
 & = & M_{0}(m_{0})\partial_{0}u+Au-2\rho M_{0}(m_{0})u\\
 & = & -(M_{0}(m_{0})\partial_{0}^{*}-A)u,
\end{eqnarray*}
 where we have used that $\partial_{0}^{*}=-\partial_{0}+2\rho,$
which can be verified immediately.\end{proof}
\begin{cor}
\label{cor:commutator} Let $\epsilon,\rho>0$. Assume that $M_{0}$
satisfies the properties (\ref{Lipschitz}) and (\ref{differentiable}).
Then for $u\in H_{\rho,0}(\mathbb{R};H)$ we have that 
\[
(1+\epsilon\partial_{0})^{-1}\partial_{0}M_{0}(m_{0})u=\partial_{0}M_{0}(m_{0})(1+\epsilon\partial_{0})^{-1}u-\epsilon\partial_{0}(1+\epsilon\partial_{0})^{-1}\dot{M}_{0}(m_{0})(1+\epsilon\partial_{0})^{-1}u.
\]
\end{cor}
\begin{proof}
For $u\in H_{\rho,0}(\mathbb{R};H$), we compute, invoking the product
rule (\ref{eq:product_rule}), that 
\begin{align*}
 & \left((1+\epsilon\partial_{0})^{-1}\partial_{0}M_{0}(m_{0})-\partial_{0}M_{0}(m_{0})(1+\epsilon\partial_{0})^{-1}\right)u\\
 & =(1+\epsilon\partial_{0})^{-1}\left(\partial_{0}M_{0}(m_{0})(1+\epsilon\partial_{0})-(1+\epsilon\partial_{0})\partial_{0}M_{0}(m_{0})\right)(1+\epsilon\partial_{0})^{-1}u\\
 & =(1+\epsilon\partial_{0})^{-1}\left(\partial_{0}M_{0}(m_{0})\epsilon\partial_{0}-\epsilon\partial_{0}\left(\partial_{0}M_{0}(m_{0})\right)\right)(1+\epsilon\partial_{0})^{-1}u\\
 & =(1+\epsilon\partial_{0})^{-1}\epsilon\partial_{0}\left(M_{0}(m_{0})\partial_{0}-\partial_{0}M_{0}(m_{0})\right)(1+\epsilon\partial_{0})^{-1}u\\
 & =\epsilon\partial_{0}(1+\epsilon\partial_{0})^{-1}(-\dot{M}_{0}(m_{0}))(1+\epsilon\partial_{0})^{-1}u.\tag*{\qedhere}
\end{align*}

\end{proof}
In the spirit of the solution theory in \cite[Chapter 6]{Picard_McGhee},
we will require the following positive definiteness constraint on
the operators $M_{0},\dot{M}_{0}$ and $M_{1}$: there exists a set
$N_{1}\subseteq\mathbb{R}$ of measure zero with $N\subseteq N_{1}$
such that
\begin{equation}
\exists c_{0}>0,\rho_{0}>0\:\forall t\in\mathbb{R}\setminus N_{1},\rho\geq\rho_{0}:\rho M_{0}(t)+\frac{1}{2}\dot{M}_{0}(t)+\Re M_{1}(t)\geq c_{0}.\label{eq:pos_def}
\end{equation}

From this we derive the following estimate.
\begin{lem}
\label{lem:pos_def_d0} Assume that $M_{0}$ satisfies properties
(\ref{selfadjoint})-(\ref{differentiable}). Assume that inequality
(\ref{eq:pos_def}) holds and let $\rho\geq\rho_{0}.$ Then for $u\in D(\partial_{0})\cap D(A)$
and $a\in\mathbb{R}$ we have that
\begin{equation}
\intop_{-\infty}^{a}\Re\langle\partial_{0}M_{0}(m_{0})u+M_{1}(m_{0})u+Au|u\rangle(t)e^{-2\rho t}\,\mathrm{d}t\geq c_{0}\intop_{-\infty}^{a}|u(t)|^{2}e^{-2\rho t}\,\mathrm{d}t.\label{eq:pos_def_1}
\end{equation}

\end{lem}
For the proof of the lemma, we need the following.
\begin{lem}
\label{lem:integr_eq} Let $\rho>0.$ Assume that $M_{0}$ satisfies
the properties (\ref{selfadjoint}), (\ref{Lipschitz}) and (\ref{differentiable}).
Then for $u\in D(\partial_{0})\cap D(A)$ and $a\in\mathbb{R}$ the
following equality holds:
\begin{align}
 & \intop_{-\infty}^{a}\Re\langle\partial_{0}M_{0}(m_{0})u+M_{1}(m_{0})u+Au|u\rangle(t)e^{-2\rho t}\,\mathrm{d}t\nonumber \\
 & =\frac{1}{2}\langle u(a)|M_{0}(a)u(a)\rangle e^{-2\rho a}+\intop_{-\infty}^{a}\rho\langle u(t)|M_{0}(t)u(t)\rangle e^{-2\rho t}\,\mathrm{d}t\label{eq:integr_equ}\\
 & \quad+\intop_{-\infty}^{a}\left\langle \left.\frac{1}{2}\dot{M}_{0}(t)u(t)+\Re M_{1}(t)u(t)\right|u(t)\right\rangle e^{-2\rho t}\,\mathrm{d}t.\nonumber 
\end{align}
\end{lem}
\begin{proof}
Let $u\in D(\partial_{0})\cap D(A)$ and $a\in\mathbb{R}.$ Note that,
since $A$ is skew-selfadjoint, 
\[
\Re\langle Au|u\rangle(t)=0
\]
for almost every $t\in\mathbb{R}.$ Hence, the left-hand side in (\ref{eq:integr_equ})
equals 
\[
\intop_{-\infty}^{a}\Re\langle\partial_{0}M_{0}(m_{0})u+M_{1}(m_{0})u|u\rangle(t)e^{-2\rho t}\,\mathrm{d}t.
\]
Using the product rule (\ref{eq:product_rule}) and the selfadjointness
of $\dot{M}_{0}(t)$ for almost every $t\in\mathbb{R}$ we get that
\begin{align*}
 & 2\intop_{-\infty}^{a}\Re\langle\partial_{0}M_{0}(m_{0})u+M_{1}(m_{0})u|u\rangle(t)e^{-2\rho t}\,\mathrm{d}t\\
 & =2\intop_{-\infty}^{a}\Re\langle M_{0}(t)\partial_{0}u(t)|u(t)\rangle e^{-2\rho t}\mbox{ d}t+2\intop_{-\infty}^{a}\langle\dot{M}_{0}(t)u(t)+\Re M_{1}(t)u(t)|u(t)\rangle e^{-2\rho t}\mbox{ d}t\\
 & =\intop_{-\infty}^{a}\Re\langle\partial_{0}u(t)|M_{0}(t)u(t)\rangle e^{-2\rho t}\mbox{ d}t\\
 & \quad+\intop_{-\infty}^{a}\Re\langle M_{0}(t)\partial_{0}u(t)|u(t)\rangle e^{-2\rho t}\mbox{ d}t+2\intop_{-\infty}^{a}\langle\dot{M}_{0}(t)u(t)+\Re M_{1}(t)u(t)|u(t)\rangle e^{-2\rho t}\mbox{ d}t\\
 & =\intop_{-\infty}^{a}\langle u(\cdot)|\left(M_{0}(m_{0})u\right)(\cdot)\rangle'(t)e^{-2\rho t}\mbox{ d}t-\intop_{-\infty}^{a}\Re\langle u(t)|\left(\partial_{0}M_{0}(m_{0})u\right)(t)\rangle e^{-2\rho t}\mbox{ d}t\\
 & \quad+\intop_{-\infty}^{a}\Re\langle M_{0}(t)\partial_{0}u(t)|u(t)\rangle e^{-2\rho t}\mbox{ d}t+2\intop_{-\infty}^{a}\langle\dot{M}_{0}(t)u(t)+\Re M_{1}(t)u(t)|u(t)\rangle e^{-2\rho t}\mbox{ d}t.
\end{align*}
Using again the product rule (\ref{eq:product_rule}) we obtain
\begin{align*}
 & -\intop_{-\infty}^{a}\Re\langle u(t)|\left(\partial_{0}M_{0}(m_{0})u\right)(t)\rangle e^{-2\rho t}\mbox{ d}t+\intop_{-\infty}^{a}\Re\langle M_{0}(t)\partial_{0}u(t)|u(t)\rangle e^{-2\rho t}\mbox{ d}t\\
 & =-\intop_{-\infty}^{a}\langle\dot{M}_{0}(t)u(t)|u(t)\rangle e^{-2\rho t}\mbox{ d}t.
\end{align*}
 Hence, we arrive at 
\begin{align*}
 & 2\intop_{-\infty}^{a}\Re\langle\partial_{0}M_{0}(m_{0})u+M_{1}(m_{0})u|u\rangle(t)e^{-2\rho t}\,\mathrm{d}t\\
 & =\intop_{-\infty}^{a}\langle u(\cdot)|\left(M_{0}(m_{0})u\right)(\cdot)\rangle'(t)e^{-2\rho t}\mbox{ d}t+2\intop_{-\infty}^{a}\left\langle \left.\frac{1}{2}\dot{M}_{0}(t)u(t)+\Re M_{1}(t)u(t)\right|u(t)\right\rangle e^{-2\rho t}\mbox{ d}t\\
 & =\langle u(a)|M_{0}(a)u(a)\rangle e^{-2\rho a}+\intop_{-\infty}^{a}2\rho\langle u(t)|M_{0}(t)u(t)\rangle e^{-2\rho t}\mbox{ d}t\\
 & \quad+2\intop_{-\infty}^{a}\left\langle \left.\frac{1}{2}\dot{M}_{0}(t)u(t)+\Re M_{1}(t)u(t)\right|u(t)\right\rangle e^{-2\rho t}\mbox{ d}t,
\end{align*}
where we have used integration by parts.
\end{proof}
~
\begin{proof}[Proof of Lemma \ref{lem:pos_def_d0}]
 Using Lemma \ref{lem:integr_eq} and the fact that $M_{0}(a)$ is
non-negative we end up with 
\begin{align*}
 & \intop_{-\infty}^{a}\Re\langle\partial_{0}M_{0}(m_{0})u+M_{1}(m_{0})u+Au|u\rangle(t)e^{-2\rho t}\,\mathrm{d}t\\
 & \geq\intop_{-\infty}^{a}\left\langle \left.\left(\rho M_{0}(t)+\frac{1}{2}\dot{M}_{0}(t)+\Re M_{1}(t)\right)u(t)\right|u(t)\right\rangle e^{-2\rho t}\mbox{ d}t\\
 & \geq c_{0}\intop_{-\infty}^{a}|u(t)|^{2}e^{-2\rho t}\mbox{ d}t.\tag*{\qedhere}
\end{align*}

\end{proof}
Our next goal is to show that (\ref{eq:pos_def_1}) also holds for
elements in $\mathcal{D}.$ For doing so, we need to approximate elements
in $\mathcal{D}$ by elements in $D(\partial_{0})\cap D(A)$ in a
suitable way.
\begin{lem}
\label{lem:resolvent} For each $u\in H_{\rho,0}(\mathbb{R};H)$ we
have that $(1+\varepsilon\partial_{0})^{-1}u\to u$ as $\varepsilon\to0+.$ \end{lem}
\begin{proof}
Since the operator family $\left((1+\varepsilon\partial_{0})^{-1}\right)_{\epsilon>0}$
is uniformly bounded, it suffices to note that 
\[
(1+\varepsilon\partial_{0})^{-1}u-u=(1+\varepsilon\partial_{0})^{-1}\epsilon\partial_{0}u\to0
\]
as $\epsilon\to0+$ for every $u\in H_{\rho,1}\left(\mathbb{R};H\right).$ \end{proof}
\begin{rem}
It should be noted that literally the same result holds true for $\partial_{0}$
replaced by $\partial_{0}^{*}$. The proof follows with obvious modifications. \end{rem}
\begin{lem}
\label{lem:regularizing} Let $\epsilon>0$ and let $u\in\mathcal{D}$.
Then $(1+\varepsilon\partial_{0})^{-1}u\in D(\partial_{0})\cap D(A)$
and the following formula holds
\begin{align}
 & \left(1+\epsilon\partial_{0}\right)^{-1}\left(\partial_{0}M_{0}\left(m_{0}\right)+M_{1}\left(m_{0}\right)+A\right)u\nonumber \\
 & =\left(\partial_{0}M_{0}\left(m_{0}\right)+M_{1}\left(m_{0}\right)+A\right)\left(1+\epsilon\partial_{0}\right)^{-1}u-\epsilon\partial_{0}\left(1+\epsilon\partial_{0}\right)^{-1}\dot{M}_{0}\left(m_{0}\right)\left(1+\epsilon\partial_{0}\right)^{-1}u\nonumber \\
 & \quad+\left(1+\epsilon\partial_{0}\right)^{-1}M_{1}\left(m_{0}\right)u-M_{1}(m_{0})\left(1+\epsilon\partial_{0}\right)^{-1}u.\label{eq:commutator_with_whole_operator}
\end{align}

Moreover, we have 
\[
\left(\partial_{0}M_{0}\left(m_{0}\right)+M_{1}\left(m_{0}\right)+A\right)\left(1+\epsilon\partial_{0}\right)^{-1}u\rightharpoondown\left(\partial_{0}M_{0}\left(m_{0}\right)+M_{1}\left(m_{0}\right)+A\right)u\quad(\varepsilon\to0+)
\]

in $H_{\rho,0}(\mathbb{R};H).$\end{lem}
\begin{proof}
With the help of Corollary \ref{cor:commutator} and the fact that
$A$ and $(1+\varepsilon\partial_{0})^{-1}$ commute, the formula
(\ref{eq:commutator_with_whole_operator}) follows. From (\ref{eq:commutator_with_whole_operator})
we read off that $(1+\epsilon\partial_{0})^{-1}u\in\mathcal{D}$.
Moreover, since $\left(1+\epsilon\partial_{0}\right)^{-1}u\in H_{\rho,1}\left(\mathbb{R};H\right)$
we have $\partial_{0}M_{0}\left(m_{0}\right)\left(1+\epsilon\partial_{0}\right)^{-1}u\in H_{\rho,0}\left(\mathbb{R};H\right)$.
Defining 
\[
\left(\partial_{0}M_{0}\left(m_{0}\right)+M_{1}\left(m_{0}\right)+A\right)\left(1+\epsilon\partial_{0}\right)^{-1}u\eqqcolon F\in H_{\rho,0}\left(\mathbb{R};H\right),
\]
we get that
\[
F-\left(\partial_{0}M_{0}\left(m_{0}\right)+M_{1}\left(m_{0}\right)\right)\left(1+\epsilon\partial_{0}\right)^{-1}u=A\left(1+\epsilon\partial_{0}\right)^{-1}u\in H_{\rho,0}\left(\mathbb{R};H\right).
\]
Hence, $\left(1+\epsilon\partial_{0}\right)^{-1}u\in D\left(A\right).$
According to Lemma \ref{lem:resolvent} the left-hand side of equation
(\ref{eq:commutator_with_whole_operator}) converges to $\left(\partial_{0}M_{0}\left(m_{0}\right)+M_{1}\left(m_{0}\right)+A\right)u$
and the last two terms on the right hand side cancel out as $\varepsilon\to0+.$
Moreover, since $\left(\epsilon\partial_{0}\left(1+\epsilon\partial_{0}\right)^{-1}\dot{M}_{0}\left(m_{0}\right)\left(1+\epsilon\partial_{0}\right)^{-1}u\right)_{\varepsilon>0}$
is bounded in $H_{\rho,0}(\mathbb{R};H)$, there exists a weakly convergent
subsequence. Using that $\epsilon\partial_{0}\left(1+\epsilon\partial_{0}\right)^{-1}\dot{M}_{0}\left(m_{0}\right)\left(1+\epsilon\partial_{0}\right)^{-1}u\to0$
in $H_{\rho,-1}(\mathbb{R};H)$, we deduce that 
\[
\epsilon\partial_{0}\left(1+\epsilon\partial_{0}\right)^{-1}\dot{M}_{0}\left(m_{0}\right)\left(1+\epsilon\partial_{0}\right)^{-1}u\rightharpoondown0\quad(\varepsilon\to0+),
\]
and thus 
\[
\left(\partial_{0}M_{0}\left(m_{0}\right)+M_{1}\left(m_{0}\right)+A\right)\left(1+\epsilon\partial_{0}\right)^{-1}u\rightharpoondown\left(\partial_{0}M_{0}\left(m_{0}\right)+M_{1}\left(m_{0}\right)+A\right)u\quad(\varepsilon\to0+).\tag*{\qedhere}
\]

\end{proof}
The most important step to generalize the statement of Lemma \ref{lem:pos_def_d0}
to the case of elements in $\mathcal{D}$ is the following result.
\begin{lem}
\label{lem:pos_def_D} Assume that $M_{0}$ satisfies the properties
(\ref{Lipschitz}) and (\ref{differentiable}). Let $\rho>0$, $a\in\mathbb{R}\cup\{\infty\}$
and let $G:H_{\rho,0}(\mathbb{R};H)\to\mathbb{R}$ be continuous.
Moreover, assume that for $u\in D(\partial_{0})\cap D(A)$ we have
that 
\[
\intop_{-\infty}^{a}\Re\langle\left(\partial_{0}M_{0}(m_{0})+M_{1}(m_{0})+A\right)u|u\rangle(t)e^{-2\rho t}\,\mathrm{d}t\geq G(u).
\]
Then the latter inequality holds for all $u\in\mathcal{D}$.\end{lem}
\begin{proof}
Let $u\in\mathcal{D}.$ According to Lemma \ref{lem:regularizing}
we have that
\begin{align*}
 & \langle\X_{]-\infty,a]}(m_{0})u|(\partial_{0}M_{0}(m_{0})+M_{1}(m_{0})+A)u\rangle_{H_{\rho,0}(\mathbb{R};H)}\\
 & =\underset{\epsilon\to0}{\lim}\langle\X_{]-\infty,a]}(m_{0})(1+\epsilon\partial_{0})^{-1}u|(\partial_{0}M_{0}(m_{0})+M_{1}(m_{0})+A)(1+\epsilon\partial_{0})^{-1}u\rangle_{H_{\rho,0}(\mathbb{R};H)},
\end{align*}
where we have used that the multiplication operator $\X_{]-\infty,a]}(m_{0})$
with the cut-off function $\X_{]-\infty,a]}$ is a bounded operator
on $H_{\rho,0}(\mathbb{R};H)$ and that $(1+\epsilon\partial_{0})^{-1}$
converges strongly to $1$, by Lemma \ref{lem:resolvent}. With the
assumed inequality and the fact that for $\epsilon>0$ we have $(1+\epsilon\partial_{0})^{-1}u\in D(\partial_{0})\cap D(A)=H_{\rho,1}(\mathbb{R};H)\cap H_{\rho,0}(\mathbb{R};H_{1}(A+1)),$
by Lemma \ref{lem:regularizing}, we obtain that 
\begin{align*}
 & \Re\langle\X_{]-\infty,a]}(m_{0})u|(\partial_{0}M_{0}(m_{0})+M_{1}(m_{0})+A)u\rangle_{H_{\rho,0}(\mathbb{R};H)}\\
 & =\lim_{\varepsilon\to0}\Re\langle\X_{]-\infty,a]}(m_{0})(1+\epsilon\partial_{0})^{-1}u|(\partial_{0}M_{0}(m_{0})+M_{1}(m_{0})+A)(1+\epsilon\partial_{0})^{-1}u\rangle_{H_{\rho,0}(\mathbb{R};H)}\\
 & \geq\underset{\epsilon\to0}{\lim}\, G((1+\epsilon\partial_{0})^{-1}u)=G(u).\tag*{\qedhere}
\end{align*}
\end{proof}
\begin{cor}
\label{cor:strict_pos_def} Assume that $M_{0}$ satisfies properties
(\ref{selfadjoint})-(\ref{differentiable}). Assume that inequality
(\ref{eq:pos_def}) holds and let $\rho\geq\rho_{0}.$ Then for $u\in\mathcal{D}$
and $a\in\mathbb{R}\cup\{\infty\}$ we have that
\[
\intop_{-\infty}^{a}\Re\langle(\partial_{0}M_{0}(m_{0})+M_{1}(m_{0})+A)u|u\rangle(t)e^{-2\rho t}\,\mathrm{d}t\geq c_{0}\intop_{-\infty}^{a}|u(t)|^{2}e^{-2\rho t}\,\mathrm{d}t.
\]
\end{cor}
\begin{proof}
The statement is immediate from the Lemmas \ref{lem:pos_def_d0} and
\ref{lem:pos_def_D} with 
\[
G(u)=c_{0}\intop_{-\infty}^{a}|u(t)|^{2}e^{-2\rho t}\,\mathrm{d}t=c_{0}\langle\X_{]-\infty,a]}(m_{0})u|u\rangle_{H_{\rho,0}(\mathbb{R};H)}.\tag*{\qedhere}
\]
\end{proof}
\begin{lem}
\label{lem:adjoint-pos_def} Assume that $M_{0}$ satisfies the properties
(\ref{selfadjoint})-(\ref{differentiable}) and that inequality (\ref{eq:pos_def})
holds. Let $\rho\geq\rho_{0}$ and $u\in\left\{ v\in H_{\rho,0}(\mathbb{R};H)\,|\,(M_{0}(m_{0})\partial_{0}^{*}-A)v\in H_{\rho,0}(\mathbb{R};H)\right\} $.
Then the inequality
\[
\intop_{\mathbb{R}}\Re\langle\left(M_{0}(m_{0})\partial_{0}^{*}+M_{1}(m_{0})^{*}-A\right)u|u\rangle(t)e^{-2\rho t}\,\mathrm{d}t\geq c_{0}\intop_{\mathbb{R}}|u(t)|^{2}e^{-2\rho t}\,\mathrm{d}t
\]
holds.\end{lem}
\begin{proof}
By Corollary \ref{cor:domain-of-adjoint}, we deduce that $u\in\mathcal{D}$
and that for $a\in\mathbb{R}$ we have
\begin{align}
 & \intop_{-\infty}^{a}\Re\langle\left(M_{0}(m_{0})\partial_{0}^{*}+M_{1}(m_{0})^{*}-A\right)u|u\rangle(t)e^{-2\rho t}\,\mathrm{d}t\nonumber \\
 & =\intop_{-\infty}^{a}\Re\langle(-\partial_{0}M_{0}(m_{0})+2\rho M_{0}(m_{0})+\dot{M}_{0}(m_{0})+M_{1}(m_{0})^{*}-A)u|u\rangle(t)e^{-2\rho t}\,\mathrm{d}t.\label{eq:adjoint_reformu}
\end{align}

With Lemma \ref{lem:integr_eq} we get for $u\in D(\partial_{0})\cap D(A)$
and $a\in\mathbb{R}$ that 
\begin{align*}
 & \intop_{-\infty}^{a}\Re\langle(-\partial_{0}M_{0}(m_{0})+2\rho M_{0}(m_{0})+\dot{M}_{0}(m_{0})+M_{1}(m_{0})^{*}-A)u|u\rangle(t)e^{-2\rho t}\,\mathrm{d}t\\
 & =-\frac{1}{2}\langle u(a)|M_{0}(a)u(a)\rangle e^{-2\rho a}-\intop_{-\infty}^{a}\rho\langle u(t)|M_{0}(t)u(t)\rangle e^{-2\rho t}\mbox{ d}t\\
 & \quad-\intop_{-\infty}^{a}\left\langle \left.\frac{1}{2}\dot{M}_{0}(t)u(t)-\Re M_{1}(t)u(t)\right|u(t)\right\rangle e^{-2\rho t}\mbox{ d}t\\
 & \quad+\intop_{-\infty}^{a}\langle2\rho M_{0}(t)u(t)+\dot{M}_{0}(t)u(t)|u(t)\rangle e^{-2\rho t}\mbox{ d}t\\
 & =-\frac{1}{2}\langle u(a)|M_{0}(a)u(a)\rangle e^{-2\rho a}+\intop_{-\infty}^{a}\rho\langle u(t)|M_{0}(t)u(t)\rangle e^{-2\rho t}\mbox{ d}t\\
 & \quad+\intop_{-\infty}^{a}\left\langle \left.\frac{1}{2}\dot{M}_{0}(t)u(t)+\Re M_{1}(t)u(t)\right|u(t)\right\rangle e^{-2\rho t}\mbox{ d}t\\
 & \geq-\frac{1}{2}\langle u(a)|M_{0}(a)u(a)\rangle e^{-2\rho a}+c_{0}\intop_{-\infty}^{a}\langle u(t)|u(t)\rangle e^{-2\rho t}\mbox{ d}t
\end{align*}
Letting $a\to\infty,$ we deduce that 
\begin{align*}
 & \intop_{\mathbb{R}}\Re\langle(-\partial_{0}M_{0}(m_{0})+2\rho M_{0}(m_{0})+\dot{M}_{0}(m_{0})+M_{1}(m_{0})^{*}-A)u|u\rangle(t)e^{-2\rho t}\,\mathrm{d}t\\
 & \geq c_{0}\intop_{\mathbb{R}}\langle u(t)|u(t)\rangle e^{-2\rho t}\mbox{ d}t.
\end{align*}
Now, Lemma \ref{lem:pos_def_D} implies the latter inequality to hold
for all $u\in\mathcal{D}.$ The assertion follows from equation (\ref{eq:adjoint_reformu}).\end{proof}
\begin{thm}[Solution Theory]
\label{thm:Solutiontheory} Let $A\colon D(A)\subseteq H\to H$ be
skew-selfadjoint and $M_{0},M_{1}\in L_{s}^{\infty}(\mathbb{R};L(H)).$
Furthermore, assume that $M_{0}$ satisfies the hypotheses (\ref{selfadjoint})-(\ref{differentiable})
and that (\ref{eq:pos_def}) holds. Then the operator $\partial_{0}M_{0}\left(m_{0}\right)+M_{1}\left(m_{0}\right)+A$
is continuously invertible in $H_{\rho,0}(\mathbb{R};H)$ for each
$\rho\geq\rho_{0}$. A norm bound for the inverse is $1/c_{0}$. Moreover,
we get that 
\[
\left(\partial_{0}M_{0}\left(m_{0}\right)+M_{1}\left(m_{0}\right)+A\right)^{*}=\left(M_{0}\left(m_{0}\right)\partial_{0}^{*}+M_{1}\left(m_{0}\right)^{*}-A\right),
\]
where the latter operator is considered in $H_{\rho,0}(\mathbb{R};H)$
with maximal domain.\end{thm}
\begin{proof}
Let $\rho\geq\rho_{0}.$ By Corollary \ref{cor:strict_pos_def} we
have that 
\begin{align*}
\Re\langle u|(\partial_{0}M_{0}(m_{0})+M_{1}(m_{0})+A)u\rangle_{H_{\rho,0}(\mathbb{R};H)} & \geq c_{0}\langle u|u\rangle_{H_{\rho,0}(\mathbb{R};H)}
\end{align*}
for all $u\in\mathcal{D}$. This implies that the operator $\partial_{0}M_{0}\left(m_{0}\right)+M_{1}\left(m_{0}\right)+A$
has a continuous inverse with operator norm less than or equal to
the constant $1/c_{0}$. \\
It remains to show that $\partial_{0}M_{0}(m_{0})+M_{1}(m_{0})+A$
maps onto $H_{\rho,0}(\mathbb{R};H)$. For this we compute the adjoint
of 
\[
\mathcal{B}\coloneqq\left(\partial_{0}M_{0}\left(m_{0}\right)+M_{1}\left(m_{0}\right)+A\right)
\]
considered as an operator in $H_{\rho,0}(\mathbb{R};H)$. Let $f\in D(\mathcal{B}^{*})\subseteq H_{\rho,0}(\mathbb{R};H)$.
Then for all $u\in\mathcal{D}$ and $\epsilon>0$ we obtain with the
help of equation (\ref{eq:commutator_with_whole_operator}) that 
\begin{align}
 & \left\langle \mathcal{B}u\left|\left((1+\epsilon\partial_{0})^{-1}\right)^{*}f\right.\right\rangle _{H_{\rho,0}(\mathbb{R};H)}\nonumber \\
 & =\left\langle \left.(1+\epsilon\partial_{0})^{-1}\mathcal{B}u\right|f\right\rangle _{H_{\rho,0}(\mathbb{R};H)}\nonumber \\
 & =\left\langle \left.\mathcal{B}(1+\epsilon\partial_{0})^{-1}u\right|f\right\rangle _{H_{\rho,0}(\mathbb{R};H)}+\langle-\epsilon\partial_{0}\left(1+\epsilon\partial_{0}\right)^{-1}\dot{M}_{0}\left(m_{0}\right)\left(1+\epsilon\partial_{0}\right)^{-1}u|f\rangle_{H_{\rho,0}(\mathbb{R};H)}\nonumber \\
 & \quad+\left\langle \left.\left(1+\epsilon\partial_{0}\right)^{-1}M_{\text{1}}\left(m_{0}\right)u-M_{1}(m_{0})\left(1+\epsilon\partial_{0}\right)^{-1}u\right|f\right\rangle _{H_{\rho,0}(\mathbb{R};H)}\nonumber \\
 & =\left\langle \left.\left(1+\epsilon\partial_{0}\right)^{-1}u\right|\mathcal{B}^{*}f\right\rangle _{H_{\rho,0}(\mathbb{R};H)}+\left\langle u\left|\big(-\epsilon\partial_{0}\left(1+\epsilon\partial_{0}\right)^{-1}\dot{M}_{0}\left(m_{0}\right)\left(1+\epsilon\partial_{0}\right)^{-1}\big)^{*}f\right.\right\rangle _{H_{\rho,0}(\mathbb{R};H)}\nonumber \\
 & \quad+\left\langle u\left|\left(\left(1+\epsilon\partial_{0}\right)^{-1}M_{\text{1}}\left(m_{0}\right)-M_{1}(m_{0})\left(1+\epsilon\partial_{0}\right)^{-1}\right)^{*}f\right.\right\rangle _{H_{\rho,0}(\mathbb{R};H)}.\label{eq:adjoint_B_reg}
\end{align}
Hence, we deduce that $(1+\epsilon\partial_{0}^{*})^{-1}[D(\mathcal{B}^{*})]\subseteq D(\mathcal{B}^{*})$
for $\epsilon>0$. %
\begin{comment}
Moreover, it is easy to see that $\left((1+\epsilon\partial_{0})^{-1}\right)^{*}=(1+\epsilon\partial_{0}^{*})^{-1}\to1$
strongly as $\epsilon\to0+$. Thus, we deduce that $\lim_{\epsilon\to0}\left(1+\epsilon\partial_{0}^{*}\right)^{-1}F\in D(\mathcal{B}^{*})$
and $\mathcal{B}^{*}\lim_{\epsilon\to0}\left(\left(1+\epsilon\partial_{0}\right)^{-1}\right)^{*}F=\lim_{\epsilon\to0}\left(\left(1+\epsilon\partial_{0}\right)^{-1}\right)^{*}\mathcal{B}^{*}F.$
\end{comment}
Moreover, we have $(1+\epsilon\partial_{0}^{*})^{-1}f\in D(\partial_{0}^{*})$.
Thus, for $u\in H_{\rho,1}(\mathbb{R};H_{1}(A+1))\subseteq D(\mathcal{B})$
and $\epsilon>0$ we get that 
\begin{align*}
 & \left\langle u\left|\mathcal{B}^{*}(1+\epsilon\partial_{0}^{*})^{-1}f\right.\right\rangle _{H_{\rho,0}(\mathbb{R};H)}\\
 & =\left\langle \mathcal{B}u\left|(1+\epsilon\partial_{0}^{*})^{-1}f\right.\right\rangle _{H_{\rho,0}(\mathbb{R};H)}\\
 & =\left\langle \left.\left(\partial_{0}M_{0}(m_{0})+M_{1}(m_{0})+A\right)u\right|(1+\epsilon\partial_{0}^{*})^{-1}f\right\rangle _{H_{\rho,0}(\mathbb{R};H)}\\
 & =\left\langle u\left|\left(M_{0}(m_{0})\partial_{0}^{*}+M_{1}(m_{0})^{*}\right)(1+\epsilon\partial_{0}^{*})^{-1}f\right.\right\rangle _{H_{\rho,0}(\mathbb{R};H)}+\left\langle Au\left|(1+\epsilon\partial_{0}^{*})^{-1}f\right.\right\rangle _{H_{\rho,0}(\mathbb{R};H)}.
\end{align*}
Since $H_{\rho,1}(\mathbb{R};H_{1}(A+1))$ is a core for $A$, we
deduce that $(1+\epsilon\partial_{0}^{*})^{-1}f\in D(A)$ for $\epsilon>0$.
Moreover, we have
\[
\mathcal{B}^{*}\left(1+\epsilon\partial_{0}^{*}\right)^{-1}f=\left(M_{0}(m_{0})\partial_{0}^{*}+M_{1}(m_{0})^{*}-A\right)\left(1+\epsilon\partial_{0}^{*}\right)^{-1}f.
\]
Using (\ref{eq:adjoint_B_reg}) we can estimate
\begin{align*}
\left|\mathcal{B}^{\ast}(1+\epsilon\partial_{0}^{\ast})^{-1}f\right|_{H_{\rho,0}(\mathbb{R};H)} & =\sup\left\{ \left.\left|\left\langle u\left|\mathcal{B}^{\ast}(1+\epsilon\partial_{0}^{\ast})^{-1}f\right.\right\rangle _{H_{\rho,0}(\mathbb{R};H)}\right|\,\right|\, u\in\mathcal{D},|u|_{H_{\rho,0}(\mathbb{R};H)}\leq1\right\} \\
 & \leq|\mathcal{B}^{\ast}f|_{H_{\rho,0}(\mathbb{R};H)}+2|M_{0}|_{\mathrm{Lip}}|f|_{H_{\rho,0}(\mathbb{R};H)}\\
 & \quad+2\|M_{1}(m_{0})\|_{L(H_{\rho,0}(\mathbb{R};H))}|f|_{H_{\rho,0}(\mathbb{R};H)}
\end{align*}
for every $\varepsilon>0$ and thus, we find a weakly convergent subsequence
in $H_{\rho,0}(\mathbb{R};H)\subseteq H_{\rho,-1}(\mathbb{R};H)\cap H_{\rho,0}(\mathbb{R};H_{-1}(A+1))$.
Moreover, note that $\left(M_{0}(m_{0})\partial_{0}^{*}+M_{1}(m_{0})^{*}-A\right)\left(1+\epsilon\partial_{0}^{*}\right)^{-1}f$
converges to $\left(M_{0}(m_{0})\partial_{0}^{*}+M_{1}(m_{0})^{*}-A\right)f$
in $H_{\rho,-1}(\mathbb{R};H)\cap H_{\rho,0}(\mathbb{R};H_{-1}(A+1)).$
Thus, by the (weak) closedness of $\mathcal{B}^{\ast}$ we derive
\[
\mathcal{B}^{*}f=\left(M_{0}(m_{0})\partial_{0}^{*}+M_{1}(m_{0})^{*}-A\right)f
\]
 and 
\[
D(\mathcal{B}^{*})\subseteq\left\{ f\in H_{\rho,0}(\mathbb{R};H)\,\left|\,\left(M_{0}(m_{0})\partial_{0}^{*}+M_{1}(m_{0})^{*}-A\right)f\in H_{\rho,0}(\mathbb{R};H)\right.\right\} .
\]
We define 
\begin{align*}
\mathcal{C}\colon D(\mathcal{C})\subseteq H_{\rho,0}(\mathbb{R};H) & \to H_{\rho,0}(\mathbb{R};H)\\
f & \mapsto\left(M_{0}(m_{0})\partial_{0}^{*}+M_{1}(m_{0})^{*}-A\right)f,
\end{align*}
where $D(\mathcal{C})\coloneqq\left\{ f\in H_{\rho,0}(\mathbb{R};H)\,|\,\left(M_{0}(m_{0})\partial_{0}^{*}-A\right)f\in H_{\rho,0}(\mathbb{R};H)\right\} .$
Lemma \ref{lem:adjoint-pos_def} ensures that $\mathcal{C}$ is one-to-one.
Thus, so is $\mathcal{B}^{*}$. According to the projection theorem
we have the orthogonal decomposition
\begin{align*}
H_{\rho,0}(\mathbb{R};H)= & N((\partial_{0}M_{0}(m_{0})+M_{1}(m_{0})+A)^{*})\oplus R(\partial_{0}M_{0}(m_{0})+M_{1}(m_{0})+A)\\
= & \{0\}\oplus R(\partial_{0}M_{0}(m_{0})+M_{1}(m_{0})+A)
\end{align*}
and this establishes the onto-property of $\partial_{0}M_{0}(m_{0})+M_{1}(m_{0})+A$.
Moreover, we get that $\left(\mathcal{B}^{-1}\right)^{*}=\left(\mathcal{B}^{*}\right)^{-1}\subseteq\mathcal{C}^{-1}$.
The first operator is left-total. Thus, $\mathcal{B}^{*}=\mathcal{C}$. 
\end{proof}
\begin{comment}
Using the contraction mapping theorem the following assertion is an
immediate corollary of the latter theorem.
\begin{cor}
Let the assumption of Theorem \ref{thm:Solutiontheory} hold and let
$M_{\infty}$ be a causal operator in $L\left(H_{\rho,0}\left(\mathbb{R},H\right)\right)$
generated as the closure of an operator 
\[
\tilde{M}_{\infty}:\interior C_{\infty}\left(\mathbb{R},H\right)\to\bigcap_{\rho\in\mathbb{R}_{\geq\rho_{0}}}H_{\rho,0}\left(\mathbb{R},H\right)
\]
for some $\rho_{0}\in\mathbb{R}_{>0}$ satisfying
\[
\limsup_{\rho}\left\Vert M_{\infty}\right\Vert _{L\left(H_{\rho,0}\left(\mathbb{R},H\right)\right)}<c_{0}.
\]

Then the operator $\partial_{0}M_{0}\left(m_{0}\right)+M_{1}\left(m_{0}\right)+M_{\infty}+A$
with domain 
\[
\left\{ U\in H_{\rho,0}\left(\mathbb{R},H\right)|\partial_{0}M_{0}\left(m_{0}\right)U+AU\in H_{\rho,0}\left(\mathbb{R},H\right)\right\} 
\]
 is continuously invertible in $H_{\rho,0}(\mathbb{R},H)$ for each
$\rho\geq\rho_{0}$. \end{cor}
\end{comment}

\subsection*{Causality}

At first we give the definition of causality in our framework.
\begin{defn}
\label{def: causality}Let $H$ be a Hilbert space, $\rho>0$ and
$G\colon D(G)\subseteq H_{\rho,0}(\mathbb{R};H)\to H_{\rho,0}(\mathbb{R};H)$.
Then $G$ is called\emph{ (forward) causal, }if for each $a\in\mathbb{R}$
and each $f,g\in D(G)$ the implication 
\[
\X_{]-\infty,a]}(m_{0})(f-g)=0\Longrightarrow\X_{]-\infty,a]}(m_{0})\left(G(f)-G(g)\right)=0
\]
holds. 
\end{defn}
Now, we want to show that our solution operator $(\partial_{0}M_{0}(m_{0})+M_{1}(m_{0})+A)^{-1}$
is causal in $H_{\rho,0}(\mathbb{R};H)$.

\begin{thm}[causal solution operator]
 Under the assumptions of Theorem \ref{thm:Solutiontheory} the solution
operator $(\partial_{0}M_{0}(m_{0})+M_{1}(m_{0})+A)^{-1}$ is causal
in $H_{\rho,0}(\mathbb{R};H)$ for each $\rho\geq\rho_{0}$.\end{thm}
\begin{proof}
Let $f\in H_{\rho,0}(\mathbb{R};H)$ with $\X_{]-\infty,a]}(m_{0})f=0$.
We define $u\coloneqq(\partial_{0}M_{0}(m_{0})+M_{1}(m_{0})+A)^{-1}f\in\mathcal{D}$
and estimate, using Corollary \ref{cor:strict_pos_def},
\[
0=\intop_{-\infty}^{a}\Re\langle f|u\rangle(t)e^{-2\rho t}\mbox{ d}t\geq c_{0}\intop_{-\infty}^{a}|u(t)|^{2}e^{-2\rho t}\mbox{ d}t,
\]
which shows that $\X_{]-\infty,a]}(m_{0})u=0$. Thus, due to linearity,
the solution operator is causal.
\end{proof}

\subsection*{An illustrative example}

To exemplify what has been achieved so far, let us consider a somewhat
contrived and simplistic example.

The starting point of our presentation is the $\left(1+1\right)$-dimensional
wave equation
\begin{eqnarray*}
\partial_{0}^{2}u-\partial_{1}^{2}u & = & f\mbox{ on }\mathbb{R}\times\mathbb{R}.
\end{eqnarray*}
As usual we rewrite this equation as a first order system of the form
\begin{equation}
\left(\partial_{0}\left(\begin{array}{cc}
1 & 0\\
0 & 1
\end{array}\right)+\left(\begin{array}{cc}
0 & -\partial_{1}\\
-\partial_{1} & 0
\end{array}\right)\right)\left(\begin{array}{c}
u\\
v
\end{array}\right)=\left(\begin{array}{c}
\partial_{0}^{-1}f\\
0
\end{array}\right).\label{eq:wave-eq}
\end{equation}
In this case we can compute the solution by Duhamel's principle in
terms of the unitary group generated by the skew-selfadjoint operator
\[
\left(\begin{array}{cc}
0 & -\partial_{1}\\
-\partial_{1} & 0
\end{array}\right).
\]
This would be the simplest autonomous case. Let us now, based on this,
consider a slightly more complicated situation, which is, however,
still autonomous:
\begin{align}
 & \left(\partial_{0}\left(\begin{array}{cc}
\X_{\mathbb{R}\setminus\,]-\varepsilon,0[}(m_{1}) & 0\\
0 & \X_{\mathbb{R}\setminus\,]-\varepsilon,\varepsilon[}(m_{1})
\end{array}\right)+\left(\begin{array}{cc}
\X_{]-\varepsilon,0[}(m_{1}) & 0\\
0 & \X_{]-\varepsilon,\varepsilon[}(m_{1})
\end{array}\right)+\left(\begin{array}{cc}
0 & -\partial_{1}\\
-\partial_{1} & 0
\end{array}\right)\right)\left(\begin{array}{c}
u\\
v
\end{array}\right)\nonumber \\
 & =\left(\begin{array}{c}
\partial_{0}^{-1}f\\
0
\end{array}\right),\label{eq:auto_ex}
\end{align}

where $\X_{I}(m_{1})$ denotes the spatial multiplication operator
with the cut-off function $\X_{I},$ i.e. $\left(\X_{I}(m_{1})f\right)(t,x)=\X_{I}(x)f(t,x)$
for almost every $(t,x)\in\mathbb{R}\times\mathbb{R}$, every $f\in H_{\rho,0}(\mathbb{R};L^{2}(\mathbb{R}))$
and $I\subseteq\mathbb{R}$. In the notation of the previous section
we have 
\[
M_{0}\left(m_{0}\right)\coloneqq\left(\begin{array}{cc}
\X_{\mathbb{R}\setminus\,]-\varepsilon,0[}(m_{1}) & 0\\
0 & \X_{\mathbb{R}\setminus\,]-\varepsilon,\varepsilon[}(m_{1})
\end{array}\right)
\]

and

\[
M_{1}\left(m_{0}\right):=\left(\begin{array}{cc}
\X_{]-\varepsilon,0[}(m_{1}) & 0\\
0 & \X_{]-\varepsilon,\varepsilon[}(m_{1})
\end{array}\right)
\]
and both are obviously not time-dependent. Note that our solution
condition (\ref{eq:pos_def}) is satisfied and hence, according to
our findings, problem (\ref{eq:auto_ex}) is well-posed in the sense
of Theorem \ref{thm:Solutiontheory}. By the dependence of the operators
$M_{0}(m_{0})$ and $M_{1}(m_{0})$ on the spatial parameter, we see
that (\ref{eq:auto_ex}) changes its type from hyperbolic to elliptic
to parabolic and back to hyperbolic and so standard semigroup techniques
are not at hand to solve the equation. Indeed, in the subregion $]-\varepsilon,0[$
the problem reads as 

\[
\left(\begin{array}{c}
u\\
v
\end{array}\right)+\left(\begin{array}{cc}
0 & -\partial_{1}\\
-\partial_{1} & 0
\end{array}\right)\left(\begin{array}{c}
u\\
v
\end{array}\right)=\left(\begin{array}{c}
\partial_{0}^{-1}f\\
0
\end{array}\right),
\]

which may be rewritten as an elliptic equation for $u$ of the form
\[
u-\partial_{1}^{2}u=\partial_{0}^{-1}f.
\]

For the region $]0,\varepsilon[$ we get 
\[
\left(\partial_{0}\left(\begin{array}{cc}
1 & 0\\
0 & 0
\end{array}\right)+\left(\begin{array}{cc}
0 & 0\\
0 & 1
\end{array}\right)+\left(\begin{array}{cc}
0 & -\partial_{1}\\
-\partial_{1} & 0
\end{array}\right)\right)\left(\begin{array}{c}
u\\
v
\end{array}\right)=\left(\begin{array}{c}
\partial_{0}^{-1}f\\
0
\end{array}\right),
\]

which yields a parabolic equation for $u$ of the form 
\[
\partial_{0}u-\partial_{1}^{2}u=\partial_{0}^{-1}f.
\]

In the remaining subdomain $\mathbb{R}\setminus\,]-\varepsilon,\varepsilon[$
the problem is of the original form (\ref{eq:wave-eq}), which corresponds
to a hyperbolic problem for $u$. \\
To turn this into a genuinely time-dependent problem we now make a
modification to problem (\ref{eq:auto_ex}). We define the function
\[
\varphi(t)\coloneqq\begin{cases}
0 & \mbox{ if }t\leq0,\\
t & \mbox{ if }0<t\leq1,\\
1 & \mbox{ if }1<t
\end{cases}\quad(t\in\mathbb{R})
\]
and consider the material-law operator 
\[
M_{0}\left(m_{0}\right)=\varphi(m_{0})\left(\begin{array}{cc}
\X_{\mathbb{R}\setminus\,]-\varepsilon,0[}(m_{1}) & 0\\
0 & \X_{\mathbb{R}\setminus\,]-\varepsilon,\varepsilon[}(m_{1})
\end{array}\right),
\]

which now also degenerates in time. Moreover we modify $M_{1}(m_{0})$
by adding a time-dependence of the form 
\[
M_{1}(m_{0})=\left(\begin{array}{cc}
\X_{]-\infty,0[}(m_{0})+\X_{[0,\infty[}(m_{0})\X_{]-\varepsilon,0[}(m_{1}) & 0\\
0 & \X_{]-\infty,0[}(m_{0})+\X_{[0,\infty[}(m_{0})\X_{]-\varepsilon,\varepsilon[}(m_{1})
\end{array}\right).
\]
We show that this time-dependent material law still satisfies our
solvability condition. To this end let $\rho>0.$ Note that 
\[
\varphi'(t)=\begin{cases}
1 & \mbox{ if }t\in]0,1[,\\
0 & \mbox{ otherwise}
\end{cases}
\]

and thus, for $t\leq0$ we have 
\[
\rho M_{0}(t)+\frac{1}{2}\dot{M}_{0}(t)+\Re M_{1}(t)=\left(\begin{array}{cc}
1 & 0\\
0 & 1
\end{array}\right)\geq1.
\]

For $0<t\leq1$ we estimate 
\begin{align*}
 & \rho M_{0}(t)+\frac{1}{2}\dot{M}_{0}(t)+\Re M_{1}(t)\\
 & =\left(\frac{1}{2}+\rho t\right)\left(\begin{array}{cc}
\X_{\mathbb{R}\setminus\,]-\varepsilon,0[}(m_{1}) & 0\\
0 & \X_{\mathbb{R}\setminus\,]-\varepsilon,\varepsilon[}(m_{1})
\end{array}\right)+\left(\begin{array}{cc}
\X_{]-\varepsilon,0[}(m_{1}) & 0\\
0 & \X_{]-\varepsilon,\varepsilon[}(m_{1})
\end{array}\right)\geq\frac{1}{2}
\end{align*}

and, finally, for $t>1$ we obtain that 
\begin{align*}
 & \rho M_{0}(t)+\frac{1}{2}\dot{M}_{0}(t)+\Re M_{1}(t)\\
 & =\rho\left(\begin{array}{cc}
\X_{\mathbb{R}\setminus\,]-\varepsilon,0[}(m_{1}) & 0\\
0 & \X_{\mathbb{R}\setminus\,]-\varepsilon,\varepsilon[}(m_{1})
\end{array}\right)+\left(\begin{array}{cc}
\X_{]-\varepsilon,0[}(m_{1}) & 0\\
0 & \X_{]-\varepsilon,\varepsilon[}(m_{1})
\end{array}\right)\geq\min\{\rho,1\}.
\end{align*}

\begin{rem}
Note that the spatial operator $\left(\begin{array}{cc}
0 & -\partial_{1}\\
-\partial_{1} & 0
\end{array}\right)$ in the previous example can be substituted by every skew-selfadjoint
operator. In applications, it turns out that this operator typically
is a block operator matrix of the form $\left(\begin{array}{cc}
0 & C^{\ast}\\
-C & 0
\end{array}\right),$ where $C$ is a densely defined closed linear operator between two
Hilbert spaces. Indeed, even the one-dimensional transport equation
shares this form, if one decomposes the functions in their even and
odd parts (see \cite[p. 17 f.]{Picard_2012_mother}). Moreover, it
should be noted that the block structures of the operator $\left(\begin{array}{cc}
0 & C^{\ast}\\
-C & 0
\end{array}\right)$ and of the operators $M_{0}(m_{0})$ and $M_{1}(m_{0})$ need not
to be comparable (it turns out that this naturally arises in the study
of boundary control systems, cf. \cite{Picard2012-conservative},
\cite{Picard_2012_Control}). In those cases the semi-group approach
for showing well-posedness is not applicable, without further requirements
on the block structures of the involved operators.
\end{rem}

\subsection*{\label{sub:A-Perturbation-Result}Some perturbation results}

In applications, it is useful to have a perturbation result at hand.
To this end, we assume we are given a linear mapping 
\[
M_{\infty}\colon D(M_{\infty})\subseteq\bigcap_{\rho\geq\rho_{0}}H_{\rho,0}\left(\mathbb{R};H\right)\to\bigcap_{\rho\geq\rho_{0}}H_{\rho,0}\left(\mathbb{R};H\right)
\]
for some $\rho_{0}>0$ in the way that for all $\rho\geq\rho_{0}$
we have that $D(M_{\infty})\subseteq H_{\rho,0}(\mathbb{R};H)$ is
dense%
\footnote{Note that as an example $\interior C_{\infty}(\mathbb{R};H)$ is dense
in $H_{\rho,0}(\mathbb{R};H)$ for all $\rho>0$. %
} and $M_{\infty}$ considered as a mapping from $H_{\rho,0}(\mathbb{R};H)$
to $H_{\rho,0}(\mathbb{R};H)$ is continuous. The assumptions give
rise to a continuous extension, denoted with the same symbol. A straightforward
consequence of our previous findings is the following.
\begin{thm}
\label{thm:Solutiontheory-per1} Let $A\colon D(A)\subseteq H\to H$
be skew-selfadjoint and $M_{0},M_{1}\in L_{s}^{\infty}(\mathbb{R};L(H)).$
Furthermore, assume that $M_{0}$ satisfies the properties (\ref{selfadjoint})-(\ref{differentiable})
and that (\ref{eq:pos_def}) holds. Assume that 
\[
\limsup_{\rho\to\infty}\left\Vert M_{\infty}\right\Vert _{L(H_{\rho,0}(\mathbb{R};H))}<c_{0}.
\]
Then there exists $\rho_{1}>0$ such that the operator $\partial_{0}M_{0}\left(m_{0}\right)+M_{1}\left(m_{0}\right)+M_{\infty}+A$
is continuously invertible in $H_{\rho,0}(\mathbb{R};H)$ for each
$\rho\geq\rho_{1}$. If, in addition, $M_{\infty}$ is causal, then
so is $\left(\partial_{0}M_{0}\left(m_{0}\right)+M_{1}\left(m_{0}\right)+M_{\infty}+A\right)^{-1}$.\end{thm}
\begin{proof}
Let $\rho_{1}>0$ be such that $\left\Vert M_{\infty}\right\Vert _{L(H_{\rho,0}(\mathbb{R};H))}<c_{0}$
for all $\rho\ge\rho_{1}$. Let $f\in H_{\rho,0}(\mathbb{R};H)$.
Then, the mapping 
\begin{eqnarray*}
\Phi\colon H_{\rho,0}(\mathbb{R};H) & \to & H_{\rho,0}(\mathbb{R};H)\\
u & \mapsto & \left(\partial_{0}M_{0}(m_{0})+M_{1}(m_{0})+A\right)^{-1}(f-M_{\infty}u)
\end{eqnarray*}
is a strict contraction, by Theorem \ref{thm:Solutiontheory}. Observing
that $u\in H_{\rho,0}(\mathbb{R};H)$ satisfies 
\[
\left(\partial_{0}M_{0}\left(m_{0}\right)+M_{1}\left(m_{0}\right)+M_{\infty}+A\right)u=f
\]
if and only if it is a fixed point of $\Phi$, we get existence and
uniqueness of a solution with the help of the contraction mapping
principle. If $M_{\infty}$ is causal, then so is $\Phi$ as a composition
and a sum of causal mappings. Hence, $\left(\partial_{0}M_{0}\left(m_{0}\right)+M_{1}\left(m_{0}\right)+M_{\infty}+A\right)^{-1}$
is causal.\end{proof}
\begin{rem}
Note that this perturbation result applies similarly to the case of
non-linear perturbations if the best Lipschitz constant $\left|M_{\infty}\right|_{\rho,\mathrm{Lip}}$
of the perturbation $M_{\infty}$ considered as an operator in $H_{\rho,0}\left(\mathbb{R};H\right),$
$\rho\in\oi0\infty$, satisfies
\[
\limsup_{\rho\to\infty}\left|M_{\infty}\right|_{\rho,\mathrm{Lip}}<c_{0}.
\]

\end{rem}
It is possible to derive the following more sophisticated perturbation
result, which needs little more effort. We introduce the following
notation: For a closed subspace $V\subseteq H$ we denote by $\iota_{V}:V\to H$
the canonical embedding of $V$ into $H$. It turns out that then
the adjoint $\iota_{V}^{\ast}:H\to V$ is the orthogonal projection
onto $V.$ Consequently $P_{V}\coloneqq\iota_{V}\iota_{V}^{\ast}:H\to H$
becomes the orthogonal projector on $V$ and $1-P_{V}=P_{V^{\bot}}=\iota_{V^{\bot}}\iota_{V^{\bot}}^{\ast}$.
\begin{thm}
\label{thm:Solutiontheory-per2} Let $A\colon D(A)\subseteq H\to H$
be skew-selfadjoint and $M_{0},M_{1}\in L_{s}^{\infty}(\mathbb{R};L(H)).$
Furthermore, assume that $M_{0}$ satisfies the properties (\ref{selfadjoint})-(\ref{differentiable}).
Moreover, assume $t\mapsto N(M_{0}\left(t\right))$ to be time-independent,
i.e., for all $t\in\mathbb{R}$ we have
\[
N(M_{0}\left(t\right))=N(M_{0}\left(0\right))\eqqcolon V.
\]
We further assume that for some set of measure zero $N_{1}\subseteq\mathbb{R}$
the following estimates hold:
\begin{equation}
\begin{array}{l}
\exists c_{0}>0,\rho_{0}>0\,\forall t\in\mathbb{R}\setminus N_{1}:\iota_{V}^{*}\Re M_{1}(t)\iota_{V}\geq c_{0},\end{array}\label{eq:pos_def-1}
\end{equation}
and
\begin{equation}
\exists c_{1}>0\:\forall t\in\mathbb{R}\setminus N_{1}:\iota_{V^{\bot}}^{*}M_{0}(t)\iota_{V^{\bot}}\geq c_{1}.\label{eq:pos_def-2}
\end{equation}
Furthermore, assume that $\limsup_{\rho\to\infty}\left\Vert M_{\infty}\right\Vert _{L(H_{\rho,0}(\mathbb{R};H))}<\infty$
and there exist $\tilde{\rho},\varepsilon>0$ such that for all $\rho\geq\tilde{\rho}$
and $u\in D(M_{\infty})$ 
\[
\left\langle \left.\Re M_{\infty}P_{V}u\right|P_{V}u\right\rangle _{H_{\rho,0}(\mathbb{R};H)}>(\varepsilon-c_{0})\left|P_{V}u\right|_{H_{\rho,0}(\mathbb{R};H)}^{2}.
\]
Then there exists $\rho_{1}>0$ such that the operator $\partial_{0}M_{0}\left(m_{0}\right)+M_{1}\left(m_{0}\right)+M_{\infty}+A$
is continuously invertible in $H_{\rho,0}(\mathbb{R};H)$ for every
$\rho\geq\rho_{1}$. If, in addition, $M_{\infty}$ is causal, then
so is $\left(\partial_{0}M_{0}\left(m_{0}\right)+M_{1}\left(m_{0}\right)+M_{\infty}+A\right)^{-1}$.
\end{thm}
The result follows by adapting the method of proof of Theorem \ref{thm:Solutiontheory}.
The crucial estimate to conclude the proof of Theorem \ref{thm:Solutiontheory-per2}
is given in the following lemma.
\begin{lem}
\label{lem:pos_def_D-1} Let $\rho\geq\rho_{0}.$ Assume that $M_{0}$
satisfies the properties (\ref{selfadjoint})-(\ref{differentiable}),
that $t\mapsto N(M_{0}(t))$ is time-independent and that the inequalities
(\ref{eq:pos_def-1})-(\ref{eq:pos_def-2}) hold. Then for all $\epsilon\in]0,c_{0}[$
there exists $c>0$ such that for all sufficiently large $\rho$ and
for $u\in\mathcal{D}$ and $a\in\mathbb{R}$ we have that
\begin{align}
 & \intop_{-\infty}^{a}\Re\left\langle \left.\left(\partial_{0}M_{0}(m_{0})+M_{1}(m_{0})+A\right)u\right|u\right\rangle (t)e^{-2\rho t}\,\mathrm{d}t\nonumber \\
 & \geq\rho c\intop_{-\infty}^{a}\left|P_{V^{\bot}}u(t)\right|^{2}e^{-2\rho t}\,\mathrm{d}t+(c_{0}-\epsilon)\intop_{-\infty}^{a}\left|P_{V}u(t)\right|^{2}e^{-2\rho t}\,\mathrm{d}t.\label{eq:pos_def_per1}
\end{align}
 \end{lem}
\begin{proof}
In order to prove (\ref{eq:pos_def_per1}) observe that by Lemma \ref{lem:pos_def_D}
it suffices to verify the inequality for $u\in D(\partial_{0})\cap D(A)$.
Moreover, by Lemma \ref{lem:integr_eq}, we only need to estimate
\[
\frac{1}{2}\langle u(a)|M_{0}(a)u(a)\rangle e^{-2\rho a}+\intop_{-\infty}^{a}\left\langle \left.\rho M_{0}(t)u(t)+\frac{1}{2}\dot{M}_{0}(t)u(t)+\Re M_{1}(t)u(t)\right|u(t)\right\rangle e^{-2\rho t}\mbox{ d}t.
\]
Since $M_{0}(a)$ is non-negative, we are reduced to showing an estimate
for 
\[
\left\langle \left.\rho M_{0}(t)\phi+\frac{1}{2}\dot{M}_{0}(t)\phi+\Re M_{1}(t)\phi\right|\phi\right\rangle 
\]
for all $t\in\mathbb{R}$ and $\phi\in H$. Using that $\dot{M}_{0}(t)$
vanishes on $V=N(M_{0}(0))$, we get that
\begin{align*}
 & \left\langle \left.\rho M_{0}(t)\phi+\frac{1}{2}\dot{M}_{0}(t)\phi+\Re M_{1}(t)\phi\right|\phi\right\rangle \\
 & =\rho\left\langle \left.M_{0}(t)P_{V^{\bot}}\phi\right|P_{V^{\bot}}\phi\right\rangle +\frac{1}{2}\langle\dot{M}_{0}(t)P_{V^{\bot}}\phi|P_{V^{\bot}}\phi+P_{V}\phi\rangle\\
 & \quad+2\left\langle \left.\Re M_{1}(t)P_{V^{\bot}}\phi\right|P_{V}\phi\right\rangle +\left\langle \left.\Re M_{1}(t)P_{V^{\bot}}\phi\right|P_{V^{\bot}}\phi\right\rangle \\
 & \quad+\left\langle \left.\Re M_{1}(t)P_{V}\phi\right|P_{V}\phi\right\rangle \\
 & \geq\left(\rho c_{1}-\frac{1}{2}|M_{0}|_{\mathrm{Lip}}\|M_{1}(m_{0})\|_{L(H_{\rho,0}(\mathbb{R};H))}\right)\left|P_{V^{\bot}}\phi\right|^{2}\\
 & \quad-\left(|M_{0}|_{\mathrm{Lip}}+2\|M_{1}(m_{0})\|_{L(H_{\rho,0}(\mathbb{R};H))}\right)\left|P_{V^{\bot}}\phi\right|\left|P_{V}\phi\right|+c_{0}\left|P_{V}\phi\right|^{2}.
\end{align*}
The assertion follows now by applying the trivial inequality $2ab\leq\frac{1}{\delta}a^{2}+\delta b^{2}$
for $a,b,\delta>0.$
\end{proof}
$\,$
\begin{proof}[Proof of Theorem \ref{thm:Solutiontheory-per2}]
 We denote $\mathcal{B}\coloneqq\partial_{0}M_{0}(m_{0})+M_{1}(m_{0})+A+M_{\infty}$
considered as an operator in $H_{\rho,0}(\mathbb{R};H)$. Note that
the maximal domain in $H_{\rho,0}(\mathbb{R};H)$ coincides with $\mathcal{D}$.
Moreover, since $M_{\infty}$ is continuous, we have, by Theorem \ref{thm:Solutiontheory},
that 
\[
\mathcal{B}^{*}=M_{0}(m_{0})\partial_{0}^{*}+M_{1}(m_{0})^{*}-A+M_{\infty}^{*}
\]
 with domain being equal to $\mathcal{D}$, by Lemma \ref{cor:domain-of-adjoint}.
\begin{comment}
It suffices to verify that both $\mathcal{B}$ and $\mathcal{B}^{*}$
are strictly monotone. 
\end{comment}
Let $\varepsilon>0$ and choose $\tilde{\rho}$ such that $\sup_{\rho\geq\tilde{\rho}}\|M_{\infty}\|_{L(H_{\rho,0}(\mathbb{R};H))}<\infty$
and 
\[
\left\langle \left.\Re M_{\infty}P_{V}u\right|P_{V}u\right\rangle _{H_{\rho,0}(\mathbb{R};H)}\geq(\varepsilon-c_{0})\left|P_{V}u\right|_{H{}_{\rho,0}(\mathbb{R};H)}^{2}
\]
for each $\rho\geq\tilde{\rho}$ and $u\in D(M_{\infty})$. Note that
due to continuity the latter inequality also holds for every $u\in H_{\rho,0}(\mathbb{R};H).$
For $u\in\mathcal{D}$ there exists, according to Lemma \ref{lem:pos_def_D-1},
a constant $c>0$ such that for all sufficiently large $\rho\geq\tilde{\rho}$
we have 
\begin{align*}
\Re\langle\mathcal{B}u|u\rangle_{H{}_{\rho,0}(\mathbb{R};H)} & =\Re\left\langle \left.(\partial_{0}M_{0}(m_{0})+M_{1}(m_{0})+A)u\right|u\right\rangle _{H_{\rho,0}(\mathbb{R};H)}+\Re\left\langle \left.M_{\infty}u\right|u\right\rangle _{H_{\rho,0}(\mathbb{R};H)}\\
 & \geq\rho c\left|P_{V^{\bot}}u\right|_{H{}_{\rho,0}(\mathbb{R};H)}^{2}+\left(c_{0}-\frac{\varepsilon}{2}\right)\left|P_{V}u\right|_{H{}_{\rho,0}(\mathbb{R};H)}^{2}+\left(\varepsilon-c_{0}\right)\left|P_{V}u\right|_{H_{\rho,0}(\mathbb{R};H)}^{2}\\
 & \quad-2\lVert M_{\infty}\rVert_{L(H_{\rho,0}(\mathbb{R};H))}|P_{V}u|_{H{}_{\rho,0}(\mathbb{R};H)}|P_{V^{\bot}}u|_{H{}_{\rho,0}(\mathbb{R};H)}\\
 & \quad-\lVert M_{\infty}\rVert_{L(H_{\rho,0}(\mathbb{R};H))}|P_{V^{\bot}}u|_{H{}_{\rho,0}(\mathbb{R};H)}^{2}\\
 & \geq\left(\rho c-\|M_{\infty}\|_{L(H_{\rho,0}(\mathbb{R};H))}-\frac{1}{\delta}\lVert M_{\infty}\rVert_{L(H_{\rho,0}(\mathbb{R};H))}^{2}\right)\left|P_{V^{\bot}}u\right|_{H{}_{\rho,0}(\mathbb{R};H)}^{2}\\
 & \quad+\left(\frac{\varepsilon}{2}-\delta\right)\left|P_{V}u\right|_{H{}_{\rho,0}(\mathbb{R};H)}^{2}
\end{align*}
for each $0<\delta<\frac{\varepsilon}{2}$. By possibly increasing
$\rho$ such that 
\[
\rho c-\lVert M_{\infty}\rVert_{L(H_{\rho,0}(\mathbb{R};H))}-\frac{1}{\delta}\lVert M_{\infty}\rVert_{L(H{}_{\rho,0}(\mathbb{R};H))}^{2}\geq\tilde{c}>0
\]
 we deduce that for all $u\in\mathcal{D}$ the estimate 
\begin{equation}
\Re\langle\mathcal{B}u|u\rangle_{H_{\rho,0}(\mathbb{R};H)}\geq\tilde{c}|u|_{H_{\rho,0}(\mathbb{R};H)}^{2}\label{eq:B_pos}
\end{equation}
holds for all sufficiently large $\rho$. By $\Re\langle\mathcal{B}u|u\rangle_{H_{\rho,0}(\mathbb{R};H)}=\Re\left\langle \left.\mathcal{B}^{*}u\right|u\right\rangle _{H_{\rho,0}(\mathbb{R};H)}$
and $D(\mathcal{B})=\mathcal{D}=D(\mathcal{B}^{*})$, we deduce that
$\mathcal{B}^{*}$ is one-to-one. Hence, $\mathcal{B}$ is continuously
invertible and onto. For showing the causality of $\mathcal{B}^{-1}$
in case of a causal operator $M_{\infty}$ it suffices to prove, that
an inequality of the form 
\[
\Re\intop_{-\infty}^{a}\langle\mathcal{B}u|u\rangle(t)e^{-2\rho t}\mbox{ d}t\geq\tilde{c}\intop_{-\infty}^{a}|u(t)|^{2}e^{-2\rho t}\mbox{ d}t
\]
holds for every $a\in\mathbb{R}$ and $u\in\mathcal{D}.$ The latter
can be shown as above, observing that due to the causality of $M_{\infty}$
we have
\begin{align*}
\Re\langle M_{\infty}u|\X_{]-\infty,a]}(m_{0})u\rangle_{H{}_{\rho,0}(\mathbb{R};H)} & =\Re\langle\X_{]-\infty,a]}(m_{0})M_{\infty}u|\X_{]-\infty,a]}(m_{0})u\rangle_{H{}_{\rho,0}(\mathbb{R};H)}\\
 & =\Re\langle M_{\infty}\X_{]-\infty,a]}(m_{0})u|\X_{]-\infty,a]}(m_{0})u\rangle_{H{}_{\rho,0}(\mathbb{R};H)}.\tag*{\qedhere}
\end{align*}

\end{proof}

\section{An application to a Kelvin-Voigt-type model in visco-elasticity}

Although, applications are obviously abundant by simply extending
well-known autonomous problems to the time-dependent coefficient case,
we intend to give a more explicit application here to illustrate some
of the issues that may appear in the non-autonomous case. A more straight-forward
application would be for example solving Maxwell's equations in the
presence of a moving body, which reduces via suitable transformations
to solving Maxwell's equations with the body at rest but coefficients
depending on time, \cite{0516086}. Applying the above theory to this
case avoids the intricacies of Kato's method of evolution systems
employed in \cite{0516086}. As a by-product, the assumptions needed
are considerably less restrictive. 

As a more intricate application we would like to elaborate on here,
we consider a time-dependent Kelvin-Voigt material in visco-elasticity.
In \cite{pre05637182} such a material is considered in connection
with modeling a solidifying visco-elastic composite material and discussing
homogenization issues. We shall use this as a motivation to analyze
well-posedness in the presence of such a material under less restrictive
assumptions.

In this model we have the equation 
\begin{align*}
\partial_{0}\eta(m_{0})\partial_{0}u-\mathrm{Div}T & =f
\end{align*}
linking stress tensor field $T$ with the displacement vector field
$u$, accompanied by a material relation of the form
\begin{equation}
T=\left(C\left(m_{0}\right)+D\left(m_{0}\right)\partial_{0}\right)\mathcal{E}\label{eq:KV}
\end{equation}
where $\Div$ is the restriction of the tensorial divergence operator
$\dive$ to symmetric tensors of order 2 and $\mathcal{E}\coloneqq\Grad u$
with $\Grad$ denoting the symmetric part of the Jacobian matrix $d\otimes u$
of the displacement vector field $u$. The operators $C,D$ and $\eta$
are thought of as material dependent parameters. Here the case $D\left(m_{0}\right)=0$
would correspond to purely elastic behavior. Introducing $v\coloneqq\partial_{0}u$
as a new unknown we arrive, by differentiating (\ref{eq:KV}), at
\begin{align*}
\partial_{0}\eta(m_{0})v-\mathrm{Div}T & =f\\
\partial_{0}\left(C\left(m_{0}\right)+D\left(m_{0}\right)\partial_{0}\right)^{-1}T & =\mathrm{Grad}v,
\end{align*}
where we can choose $\rho$ large enough, such that 
\[
C\left(m_{0}\right)+D\left(m_{0}\right)\partial_{0}=\left(C(m_{0})\partial_{0}^{-1}+D(m_{0})\right)\partial_{0}
\]

gets boundedly invertible. Assuming for sake of definiteness vanishing
of the displacement $u$ on the boundary as a boundary condition we
obtain an evolutionary equation of the form
\begin{align}
\left(\partial_{0}\left(\begin{array}{cc}
\eta(m_{0}) & 0\\
0 & \left(C\left(m_{0}\right)+D\left(m_{0}\right)\partial_{0}\right)^{-1}
\end{array}\right)+\left(\begin{array}{cc}
0 & -\Div\\
-\interior\Grad & 0
\end{array}\right)\right)\left(\begin{array}{c}
v\\
T
\end{array}\right) & =\left(\begin{array}{c}
f\\
0
\end{array}\right),\label{eq:evo-ve}
\end{align}
where the choice of boundary condition amounts to replacing $\Grad$
by the closure $\interior\Grad$ of the restriction of $\Grad$ to
vector fields with smooth components vanishing outside of a compact
subset of $\Omega.$ The underlying Hilbert space is the subspace
$L^{2}\left(\Omega\right)^{3}\oplus L_{3\times3,\mathrm{sym}}^{2}\left(\Omega\right)$
of $L^{2}\left(\Omega\right)^{3}\oplus L^{2}\left(\Omega\right)^{3\times3}$
with its natural norm, where the second block-component space $L_{3\times3,\mathrm{sym}}^{2}\left(\Omega\right)$
denotes the restriction of $L^{2}\left(\Omega\right)^{3\times3}$
to symmetric matrices with entries in $L^{2}\left(\Omega\right).$
Note that then $\left(\begin{array}{cc}
0 & -\Div\\
-\interior\Grad & 0
\end{array}\right)$ is skew-selfadjoint (see e.g. \cite[Section 5.5.1]{Picard_McGhee}).

The operator families $(C\left(t\right))_{t\in\mathbb{R}}$ and $(D(t))_{t\in\mathbb{R}}$
are assumed to be uniformly bounded in $L_{3\times3,\mathrm{sym}}^{2}\left(\Omega\right).$
Further constraint will of course be required to satisfy the assumptions
of our solution theory above. We are led to a material law operator
of the form 
\[
\mathcal{M}\left(\partial_{0}^{-1}\right)=\left(\begin{array}{cc}
\eta\left(m_{0}\right) & 0\\
0 & \left(C\left(m_{0}\right)+D\left(m_{0}\right)\partial_{0}\right)^{-1}
\end{array}\right).
\]

To deal with the term $\left(C\left(m_{0}\right)+D\left(m_{0}\right)\partial_{0}\right)^{-1}$
we need a projection technique. For this we recall that for a closed
subspace $V$ of the underlying Hilbert space $L_{3\times3,\mathrm{sym}}^{2}\left(\Omega\right)$,
$\iota_{V}$ denotes the canonical injection of $V$ into $L_{3\times3,\mathrm{sym}}^{2}\left(\Omega\right)$.
The solution theory for the Kelvin-Voigt-type model is then summarized
in the following theorem.
\begin{thm}
Let $\Omega\subseteq\mathbb{R}^{3}$ be open and $V$ be a closed
subspace of $L_{3\times3,\mathrm{sym}}^{2}(\Omega)$. Let $C\in L_{s}^{\infty}(\mathbb{R};L(L_{3\times3,\mathrm{sym}}^{2}\left(\Omega\right)))$,
$\eta\in L_{s}^{\infty}(\mathbb{R};L(L^{2}\left(\Omega\right)^{3}))$
and $B\in L_{s}^{\infty}(\mathbb{R};L(V))$. Assume that $C,\eta$
satisfy the properties (\ref{selfadjoint})-(\ref{differentiable}).
We set
\[
D(t):=\left(\begin{array}{cc}
B(t) & 0\\
0 & 0
\end{array}\right)\in L(V\oplus V^{\bot})
\]
 for all $t\in\mathbb{R}$ and we assume the existence of $c>0$ such
that for all $t\in\mathbb{R}$ we have
\[
\Re B(t)\geq c,\quad\iota_{V^{\bot}}^{*}C(t)\iota_{V^{\bot}}\geq c,\quad\eta(t)\geq c.
\]
Then for all sufficiently large $\rho$ we have that for all $F\in H_{\rho,0}(\mathbb{R};$$L^{2}\left(\Omega\right)^{3}\oplus L_{3\times3,\mathrm{sym}}^{2}\left(\Omega\right))$
there exists a unique solution $(v,T)\in H_{\rho,0}(\mathbb{R};L^{2}\left(\Omega\right)^{3}\oplus L_{3\times3,\mathrm{sym}}^{2}\left(\Omega\right))$
of the equation 
\[
\left(\partial_{0}\left(\begin{array}{cc}
\eta(m_{0}) & 0\\
0 & \left(C\left(m_{0}\right)+D\left(m_{0}\right)\partial_{0}\right)^{-1}
\end{array}\right)+\left(\begin{array}{cc}
0 & -\Div\\
-\interior\Grad & 0
\end{array}\right)\right)\left(\begin{array}{c}
v\\
T
\end{array}\right)=F.
\]
The solution depends continuously on the data. The solution operator,
mapping any right-hand side $F$ to the corresponding solution of
the latter equation, is causal.\end{thm}
\begin{proof}
The proof rests on the perturbation result Theorem \ref{thm:Solutiontheory-per1}.
Since the top left corner in the system under consideration clearly
satisfies the solvability condition (\ref{eq:pos_def}), we only have
to discuss the lower right corner. For this we have to find a more
explicit expression for $\left(C\left(m_{0}\right)+D\left(m_{0}\right)\partial_{0}\right)^{-1}$.
An easy computation shows that
\begin{align*}
 & \left(C\left(m_{0}\right)+D\left(m_{0}\right)\partial_{0}\right)^{-1}\\
 & =\left(\begin{array}{cc}
\iota_{V}^{*}C(m_{0})\iota_{V}+B(m_{0})\partial_{0} & \iota_{V}^{*}C(m_{0})\iota_{V^{\bot}}\\
\iota_{V^{\bot}}^{*}C(m_{0})\iota_{V} & \iota_{V^{\bot}}^{*}C(m_{0})\iota_{V^{\bot}}
\end{array}\right)^{-1}\\
 & =\left(\begin{array}{cc}
1 & 0\\
-(\iota_{V^{\bot}}^{*}C(m_{0})\iota_{V^{\bot}})^{-1}\iota_{V^{\bot}}^{*}C(m_{0})\iota_{V} & 1
\end{array}\right)W(m_{0})\left(\begin{array}{cc}
1 & -\iota_{V}^{*}C(m_{0})\iota_{V^{\bot}}(\iota_{V^{\bot}}^{*}C(m_{0})\iota_{V^{\bot}})^{-1}\\
0 & 1
\end{array}\right),
\end{align*}
with{\footnotesize 
\begin{align*}
 & W(m_{0})\\
 & =\left(\begin{array}{cc}
\left(\iota_{V}^{*}C(m_{0})\iota_{V}+B(m_{0})\partial_{0}-\iota_{V}^{*}C(m_{0})\iota_{V^{\bot}}(\iota_{V^{\bot}}^{*}C(m_{0})\iota_{V^{\bot}})^{-1}\iota_{V^{\bot}}^{*}C(m_{0})\iota_{V}\right)^{-1} & 0\\
0 & (\iota_{V^{\bot}}^{*}C(m_{0})\iota_{V^{\bot}})^{-1}
\end{array}\right).
\end{align*}
 }Denoting $S(m_{0})\coloneqq\left(\begin{array}{cc}
1 & 0\\
-(\iota_{V^{\bot}}^{*}C(m_{0})\iota_{V^{\bot}})^{-1}\iota_{V^{\bot}}^{*}C(m_{0})\iota_{V} & 1
\end{array}\right)$, we see that 
\[
\left(C\left(m_{0}\right)+D\left(m_{0}\right)\partial_{0}\right)^{-1}=S(m_{0})W(m_{0})S(m_{0})^{*},
\]
 by the selfadjointness of $C(m_{0})$. Now, the top left corner of
$W(m_{0})$ may be expressed with the help of a Neumann expansion
in the following way 
\begin{align*}
 & \left(B(m_{0})\partial_{0}+\iota_{V}^{*}C(m_{0})\iota_{V}-\iota_{V}^{*}C(m_{0})\iota_{V^{\bot}}(\iota_{V^{\bot}}^{*}C(m_{0})\iota_{V^{\bot}})^{-1}\iota_{V^{\bot}}^{*}C(m_{0})\iota_{V}\right)^{-1}\\
 & =\partial_{0}^{-1}B(m_{0})^{-1}\left(1+\partial_{0}^{-1}B(m_{0})^{-1}\left(\iota_{V}^{*}C(m_{0})\iota_{V}-\iota_{V}^{*}C(m_{0})\iota_{V^{\bot}}(\iota_{V^{\bot}}^{*}C(m_{0})\iota_{V^{\bot}})^{-1}\iota_{V^{\bot}}^{*}C(m_{0})\iota_{V}\right)\right)^{-1}\\
 & =\partial_{0}^{-1}B(m_{0})^{-1}+\partial_{0}^{-1}\tilde{M}_{\infty}
\end{align*}
 for some suitable $\tilde{M}_{\infty},$ satisfying $\lVert\tilde{M}_{\infty}\rVert_{L(H_{\rho,0}(\mathbb{R};V\oplus V^{\bot}))}\to0$
as $\rho\to\infty.$ Thus, we arrive at
\begin{align}
 & \left(C\left(m_{0}\right)+D\left(m_{0}\right)\partial_{0}\right)^{-1}\nonumber \\
 & =S(m_{0})\left(\begin{array}{cc}
0 & 0\\
0 & (\iota_{V^{\bot}}^{*}C(m_{0})\iota_{V^{\bot}})^{-1}
\end{array}\right)S(m_{0})^{*}+S(m_{0})\left(\begin{array}{cc}
\partial_{0}^{-1}B(m_{0})^{-1} & 0\\
0 & 0
\end{array}\right)S(m_{0})^{*}\nonumber \\
 & \quad+S(m_{0})\left(\begin{array}{cc}
\partial_{0}^{-1}\tilde{M}_{\infty} & 0\\
0 & 0
\end{array}\right)S(m_{0})^{*}.\label{eq:Kelvin_Voigt}
\end{align}
 The first term on the right-hand side can be computed as follows:
\[
S(m_{0})\left(\begin{array}{cc}
0 & 0\\
0 & (\iota_{V^{\bot}}^{*}C(m_{0})\iota_{V^{\bot}})^{-1}
\end{array}\right)S(m_{0})^{\ast}=\left(\begin{array}{cc}
0 & 0\\
0 & (\iota_{V^{\bot}}^{*}C(m_{0})\iota_{V^{\bot}})^{-1}
\end{array}\right).
\]
For the second and third term on the right-hand side we observe that
$S\in L_{s}^{\infty}(\mathbb{R};L(V\oplus V^{\bot}))$. Moreover,
$S$ is Lipschitz-continuous. Indeed, observing that 
\begin{align*}
 & \|(\iota_{V^{\bot}}^{*}C(t)\iota_{V^{\bot}})^{-1}-(\iota_{V^{\bot}}^{*}C(s)\iota_{V^{\bot}})^{-1}\|_{L(V^{\bot})}\\
 & \leq\|(\iota_{V^{\bot}}^{*}C(t)\iota_{V^{\bot}})^{-1}\|_{L(V^{\bot})}\|\iota_{V^{\bot}}^{\ast}(C(s)-C(t))\iota_{V^{\bot}}\|_{L(V^{\bot})}\|(\iota_{V^{\bot}}^{\ast}C(s)\iota_{V^{\bot}})^{-1}\|_{L(V^{\bot})}\\
 & \leq c^{-2}|C|_{\mathrm{Lip}}|s-t|
\end{align*}
and that, using $|C|_{\infty}\coloneqq\sup_{t\in\mathbb{R}}\|C(t)\|$,
\begin{align*}
 & \|(\iota_{V^{\bot}}^{*}C(t)\iota_{V^{\bot}})^{-1}\iota_{V^{\bot}}^{*}C(t)\iota_{V}-(\iota_{V^{\bot}}^{*}C(s)\iota_{V^{\bot}})^{-1}\iota_{V^{\bot}}^{*}C(s)\iota_{V}\|_{L(V,V^{\bot})}\\
 & \leq\|(\iota_{V^{\bot}}^{*}C(t)\iota_{V^{\bot}})^{-1}\|_{L(V^{\bot})}\|\iota_{V^{\bot}}^{*}(C(t)-C(s))\iota_{V}\|_{L(V,V^{\bot})}\\
 & \quad+\|(\iota_{V^{\bot}}^{*}C(t)\iota_{V^{\bot}})^{-1}-(\iota_{V^{\bot}}^{*}C(s)\iota_{V^{\bot}})^{-1}\|_{L(V^{\bot})}\|\iota_{V^{\bot}}^{*}C(s)\iota_{V}\|_{L(V,V^{\bot})}\\
 & \leq\left(c^{-1}|C|_{\mathrm{Lip}}+c^{-2}|C|_{\mathrm{Lip}}|C|_{\infty}\right)|s-t|,
\end{align*}
for $s,t\in\mathbb{R},$ we derive the Lipschitz-continuity of $S$.
Thus, $S$ satisfies the hypothesis (\ref{differentiable}), since
$L_{3\times3,\mathrm{sym}}^{2}(\Omega)$ is separable. Hence, we can
compute the second term on the right-hand side of (\ref{eq:Kelvin_Voigt})
as follows by using the product rule (\ref{eq:product_rule})
\begin{align*}
 & S(m_{0})\left(\begin{array}{cc}
\partial_{0}^{-1}B(m_{0})^{-1} & 0\\
0 & 0
\end{array}\right)S(m_{0})^{\ast}\\
 & =\partial_{0}^{-1}\partial_{0}S(m_{0})\left(\begin{array}{cc}
\partial_{0}^{-1}B(m_{0})^{-1} & 0\\
0 & 0
\end{array}\right)S(m_{0})^{\ast}\\
 & =\partial_{0}^{-1}\dot{S}(m_{0})\left(\begin{array}{cc}
\partial_{0}^{-1}B(m_{0})^{-1} & 0\\
0 & 0
\end{array}\right)S(m_{0})^{\ast}+\partial_{0}^{-1}S(m_{0})\left(\begin{array}{cc}
B(m_{0})^{-1} & 0\\
0 & 0
\end{array}\right)S(m_{0})^{\ast}.
\end{align*}
 Defining $\tilde{M}_{0}(m_{0})\coloneqq\left(\begin{array}{cc}
0 & 0\\
0 & (\iota_{V^{\bot}}^{*}C(m_{0})\iota_{V^{\bot}})^{-1}
\end{array}\right)$, $\tilde{M}_{1}(m_{0})\coloneqq S(m_{0})\left(\begin{array}{cc}
B(m_{0})^{-1} & 0\\
0 & 0
\end{array}\right)S(m_{0})^{*}$ and $\tilde{\tilde{M}}_{\infty}\coloneqq\dot{S}(m_{0})\left(\begin{array}{cc}
\partial_{0}^{-1}B(m_{0})^{-1} & 0\\
0 & 0
\end{array}\right)S(m_{0})^{*}+S(m_{0})\left(\begin{array}{cc}
\partial_{0}^{-1}\tilde{M}_{\infty} & 0\\
0 & 0
\end{array}\right)S(m_{0})^{*},$ we get that 
\[
\partial_{0}\left(C\left(m_{0}\right)+D\left(m_{0}\right)\partial_{0}\right)^{-1}=\partial_{0}\tilde{M}_{0}(m_{0})+\tilde{M}_{1}(m_{0})+\tilde{\tilde{M}}_{\infty}.
\]
Clearly, $\tilde{\tilde{M}}_{\infty}$ can be considered as a perturbation
as in Theorem \ref{thm:Solutiontheory-per1}, since
\[
\lVert\tilde{\tilde{M}}_{\infty}\rVert_{L(H_{\rho,0}(\mathbb{R};V\oplus V^{\bot}))}\to0\quad(\rho\to\infty).
\]
Moreover, one easily obtains that $\Re\langle\tilde{M}_{1}(t)\iota_{V}\phi|\iota_{V}\phi\rangle_{V\oplus V^{\bot}}\geq\left(c/|B|_{\infty}^{2}\right)|\phi|_{V}^{2}$
for all $t\in\mathbb{R}$ and $\phi\in V$. Since $\tilde{M}_{0}(t)$
is strictly positive on $V^{\bot}$ and $\dot{\tilde{M}}_{0}(t)$
is uniformly bounded in $t$, we get estimate (\ref{eq:pos_def})
for sufficiently large $\rho_{0}$. Theorem \ref{thm:Solutiontheory-per1}
concludes the proof. \end{proof}
\begin{rem}
The assumption on the subspace $V$ in the above theorem expresses
the fact that the null space of $D\left(m_{0}\right)$ is non-trivial
due to degeneracies in various regions in $\Omega$, but it is implied
that these regions do not vary in time. In other words there may be
stationary areas in which the material exhibits purely elastic behavior
and others showing different forms of visco-elastic behavior.
\end{rem}

\end{document}